\documentclass[11pt]{amsart}

\usepackage[margin=1.1in]{geometry}
\usepackage{amsmath,amssymb,amsthm,amsfonts,mathtools}
\usepackage{enumitem}
\usepackage{mathrsfs}
\usepackage{stmaryrd}
\usepackage{hyperref}
\usepackage[capitalise,nameinlink]{cleveref}
\usepackage{verbatim}

\hypersetup{
  colorlinks=true,
  linkcolor=blue,
  citecolor=blue,
  urlcolor=blue,
  pdftitle={Elliptic matroids and modular curves},
  pdfauthor={Matthew Baker}
}

\theoremstyle{plain}
\newtheorem{theorem}{Theorem}[section]
\newtheorem{proposition}[theorem]{Proposition}
\newtheorem{lemma}[theorem]{Lemma}
\newtheorem{corollary}[theorem]{Corollary}

\theoremstyle{definition}
\newtheorem{definition}[theorem]{Definition}
\newtheorem{remark}[theorem]{Remark}

\newcommand{\C}{\mathbb{C}}
\newcommand{\F}{\mathbb{F}}
\newcommand{\G}{\mathbb{G}}
\newcommand{\N}{\mathbb{N}}
\newcommand{\PP}{\mathbb{P}}
\newcommand{\Q}{\mathbb{Q}}
\newcommand{\T}{\mathbb{T}}
\newcommand{\Z}{\mathbb{Z}}
\newcommand{\sm}{\mathrm{sm}}

\newcommand{\rk}{\mathrm{rk}}

\newcommand{\Pic}{\operatorname{Pic}}
\newcommand{\Spec}{\operatorname{Spec}}
\newcommand{\Hom}{\operatorname{Hom}}

\newcommand{\GL}{\operatorname{GL}}
\newcommand{\PGL}{\operatorname{PGL}}

\newcommand{\cO}{\mathcal O}
\newcommand{\cC}{\mathcal C}

\newcommand{\cR}{\mathcal R}
\newcommand{\bR}{\mathfrak R}
\newcommand{\cT}{\mathcal T}
\newcommand{\Fpm}{\F_1^{\pm}}
\newcommand{\Band}{\mathbf{Band}}
\newcommand{\Past}{\mathbf{Past}}
\newcommand{\etale}{\'etale}

\DeclareMathOperator{\mdeg}{mdeg}

\title[Elliptic matroids and modular curves]{Elliptic matroids and modular curves}
\author{Matthew Baker}
\date{\today}

\thanks{We thank Oliver Lorscheid for helpful discussions related to the Appendix. 
This work was partially supported by NSF grant DMS2154224 and a Simons Fellowship in Mathematics (1037306, Baker).}

\begin{document}

\begin{abstract}
For $n\ge4$, let $\cT_n$ be the rank-$3$ matroid on $\Z/n\Z$ whose
bases are the three-element non-zero-sum subsets.  Let
$X_1(n)^\circ$ denote the open subscheme of the modular curve $X_1(n)$
obtained by removing the cusps corresponding to reducible N\'eron
polygons.  For $n\ge10$, we give a purely algebraic and
incidence-theoretic proof that, for every field $k$ with
$\operatorname{char}(k)\nmid n$, there is a natural bijection between
$X_1(n)^\circ(k)$ and rescaling classes of $k$-realizations of $\cT_n$.
For $k=\C$, this recovers a theorem of Borisov and Roulleau.

We then upgrade the field-valued correspondence to an isomorphism of
schemes over $\Z[1/n]$. 
The main new ingredient is a deformation-theoretic argument
which allows us to verify the isomorphism on points valued in Artinian
local rings.  As consequences, the modular curve $X_1(n)^\circ$ acquires
a natural model over $\Z[1/n]$ as a matroid realization space, and, for
primes $p\ge11$, the non-representability of $\cT_p$ over $\Q$ is
equivalent to the prime-order case of Mazur's celebrated theorem on rational
torsion points of elliptic curves.

In an appendix, we explain how to upgrade the realization space of a matroid
from an affine scheme over $\Z$ to an affine band scheme (in the sense of Baker--Jin--Lorscheid)
over $\F_1^\pm$.
\end{abstract}

\maketitle
\setcounter{tocdepth}{1}
\tableofcontents

\section{Introduction}

\subsection{Statement of the main theorems}

Fix an integer $n\ge4$.  The \emph{elliptic matroid} $\cT_n$ is the
rank-$3$ matroid on the ground set $\Z/n\Z$ whose non-bases are the
three-element subsets $\{a,b,c\}\subset\Z/n\Z$
of distinct elements satisfying $a+b+c=0$ in $\Z/n\Z$.
The terminology comes from a natural construction using torsion points on
elliptic curves.

Let $E$ be an elliptic curve over a field $k$, with identity element $O$, and
let $P\in E(k)$ be a point of exact order $n$.  The complete linear system
$|3O|$ gives a closed immersion $\iota:E\hookrightarrow\PP^2_k$
as a smooth plane cubic, with $O$ mapping to a flex point.  For
$i\in\Z/n\Z$, set $P_i:=\iota(iP)$.

The chord--tangent law says that, for distinct $a,b,c\in\Z/n\Z$,
\[
\begin{aligned}
P_a,P_b,P_c\text{ are collinear}
&\quad\Longleftrightarrow\quad
        aP+bP+cP=O \\
&\quad\Longleftrightarrow\quad
        a+b+c=0\quad\text{in }\Z/n\Z.
\end{aligned}
\]
Thus the labelled configuration $(P_i)_{i\in\Z/n\Z}$ is a realization of
$\cT_n$ over $k$.

Choose nonzero homogeneous lifts $v_i\in k^3$ of the points $P_i$.  The
resulting $3\times n$ matrix represents the same realization if it is changed
by left multiplication by an element of $\GL_3(k)$ or if its columns are
independently multiplied by nonzero scalars.  More precisely, two matrices
$(v_i)$ and $(v_i')$ are \emph{rescaling equivalent} if
\[
        v_i'=\lambda_i g(v_i)
        \qquad(i\in\Z/n\Z)
\]
for some $g\in\GL_3(k)$ and some $\lambda_i\in k^\times$.  Changing the basis
of $H^0(E,\cO_E(3O))$, and hence the projective coordinates on the plane,
changes the matrix by left multiplication; changing the chosen homogeneous
lifts changes its columns by independent scalars.  Consequently the
rescaling class of the realization depends only on the pair $(E,P)$, not on
these choices.  Equivalently, any two plane cubic embeddings for which $O$
is a flex differ, after identifying their complete linear systems with
$|3O|$, by a projective linear transformation.

The same construction applies to an irreducible nodal plane cubic $C$.  If
$O,P\in C^{\sm}(k)$, if $O$ is chosen as the identity of the group
$C^{\sm}$, if $P$ has exact order $n$, and if $\cO_C(1)\cong\cO_C(3O)$,
then the points $P_i=iP$ again realize $\cT_n$.  Let $X_1(n)^\circ$ denote
the open locus of the compactified modular curve $X_1(n)$ on which the
universal generalized elliptic curve is irreducible.  Its geometric fibers
are smooth elliptic curves and N\'eron $1$-gons, the latter admitting an
embedding by $|3O|$ as irreducible nodal plane cubics.  If
$R_{\cT_n}(k)$ denotes the set of rescaling classes of $k$-realizations of
$\cT_n$, the preceding construction therefore gives a natural map
\[
        \beta_k:X_1(n)^\circ(k)\longrightarrow R_{\cT_n}(k).
\]

A remarkable theorem of Borisov and Roulleau asserts that, when $n\ge10$ and
$k=\C$, every rescaling class of realizations of $\cT_n$ arises in this way
\cite{BR}.  More precisely, the map $\beta_\C$ is bijective for all $n \geq 10$.
Their proof makes essential use of complex-analytic methods (e.g. weight-one modular forms).  
Our first main result extends this set-theoretic bijection to every field of characteristic prime to $n$.

\begin{theorem}[Field-valued Borisov--Roulleau correspondence]
\label{thm:intro-field}
Let $n\ge10$, and let $k$ be a field with
$\operatorname{char}(k)\nmid n$.  Then
\[
        \beta_k:X_1(n)^\circ(k)\xrightarrow{\sim}R_{\cT_n}(k)
\]
is a bijection.
\end{theorem}

The hypothesis $n\ge10$ is partly an artifact of our method. For $n\le6$, 
the realization space of $\cT_n$ over $k$ is a point, while $X_1(n)^\circ$ is a curve, so the conclusion of \Cref{thm:intro-field} 
genuinely fails in this case.
For $7\le n\le9$, it turns out that $\beta_k$ is bijective even though our reconstruction method breaks down in that case.
See \Cref{sec:small-n} for details.

In contrast with the Borisov--Roulleau proof over $\C$, our proof of
Theorem~\ref{thm:intro-field} is entirely algebraic and
incidence-theoretic and does not involve modular forms or computer-assisted calculations.  
It uses only elementary projective geometry,
B\'ezout's theorem, Chasles' theorem, and the group law
on the smooth locus of a nonsingular or nodal plane cubic curve.  The use of
Chasles propagation is inspired by the role of such configurations
in Green and Tao's work on point sets with few ordinary lines \cite{GT}.

Over $\C$, Borisov and Roulleau prove more than a bijection on closed points:
they endow the realization space with its natural affine algebraic-curve
structure and show that the modular map is an isomorphism of affine
algebraic curves \cite{BR}.  We generalize this result as follows.

Both matroid realization spaces and modular curves have natural models over $\Z$.
For a matroid $M$, its Pl\"ucker ring\footnote{We note that the Pl\"ucker ring of a matroid is a multigraded refinement of White's \emph{bracket ring} \cite{White}, localized at the basis brackets.} 
$S_M$ has a natural multigrading, and we define
\[
        A_M:=(S_M)_0,
        \qquad
        \cR_M:=\Spec A_M.
\]
For $M=\cT_n$, we put
\[
        \cR_n
        :=
        \cR_{\cT_n}\times_{\Spec\Z}\Spec\Z[1/n].
\]
On the modular side, the standard integral model of $X_1(n)$ is smooth over
$\Z[1/n]$; see \cite{DR,KM,Conrad-KM}.  The construction above globalizes to
a morphism of $\Z[1/n]$-schemes $\beta:X_1(n)^\circ\longrightarrow\cR_n$.
Our second main result is the corresponding scheme-theoretic extension of
the Borisov--Roulleau theorem.

\begin{theorem}[Scheme-theoretic Borisov--Roulleau correspondence]
\label{thm:main}
Let $n\ge10$.  The modular morphism 
\[
\beta:X_1(n)^\circ\longrightarrow\cR_n
\]
is an isomorphism of schemes over $\Z[1/n]$.
\end{theorem}

We do not work with $\Z[1/n]$-schemes, rather than $\Z$-schemes, merely for the sake of convenience: the situation over $\Z$ is genuinely more complicated, cf.~\Cref{rem:why-invert-n}.

The ordinary scheme $\cR_M$ is itself the shadow of a more general object.  
Matroid realization spaces are naturally defined as band
schemes over the initial band $\Fpm$.  In Appendix~\ref{sec:bands} we
construct a canonical realization band $R_M$ and a corresponding band scheme $\bR_M:=\Spec R_M$
whose associated scheme is $\cR_M$.  

\subsection{Overview of the proofs}

We next describe the main ideas of the proofs.  For
Theorem~\ref{thm:intro-field}, start with an arbitrary realization
$(P_i)_{i\in\Z/n\Z}\subset\PP^2(k)$
of $\cT_n$.  Six of the prescribed three-point lines form two reducible
cubics whose scheme-theoretic intersection consists of nine marked points.
The cubics through these points form a pencil, and the tenth marked point
selects a unique member $C$ of this pencil.  The central step is then a
Chasles propagation argument.  A $3\times3$ array of distinct
indices whose row sums and column sums are zero gives two reducible cubics,
one formed by its row-lines and the other by its column-lines.  The
classical Chasles theorem implies that any cubic through eight of
the nine associated marked points also contains the ninth.  Applying this
observation to a carefully chosen sequence of zero-sum grids propagates $C$
through the entire realization.

Once all marked points lie on $C$, elementary combinatorial and projective geometry
arguments show that $C$ is irreducible and that every marked point is smooth.
We then equip the smooth locus $C^{\sm}$ with the group law having identity $P_0$.  The
collinearity relations give a finite system of equations in
$\Pic^0(C)$; solving this system shows that
\[
        P_i=iP_1,
        \qquad
        \cO_C(1)\cong\cO_C(3P_0),
\]
and that $P_1$ has exact order $n$.  The assumption
$\operatorname{char}(k)\nmid n$ rules out a cuspidal cubic, since the smooth
locus of a cuspidal cubic becomes $\G_a$ after base change and therefore has
no nonzero $n$-torsion.  This reconstructs the required smooth or nodal
marked cubic.  

The proof of Theorem~\ref{thm:main} extends these reconstruction arguments
from fields to local Artinian rings.  The motivating principle is the
following criterion: a finitely presented morphism between Noetherian schemes over an affine base
$\Spec R$ is an isomorphism if and only if it induces a
bijection on $A$-valued points for every local Artinian $R$-algebra $A$.

Accordingly, let $A$ be an arbitrary local Artinian $\Z[1/n]$-algebra.
Using a distinguished projective frame, an $A$-valued point of $\cR_n$ has
a unique normalized representative.  
The seed-pencil construction extends
to this relative setting, and a relative Chasles argument, using
base change and Nakayama's lemma, propagates the relative cubic through all
of the marked sections.  
We conclude that
$\beta(A):X_1(n)^\circ(A)\xrightarrow{\sim}\cR_n(A)$
is an isomorphism for every such $A$, and the Artinian-point criterion gives
Theorem~\ref{thm:main}.  

\subsection{Structure of the paper}

The paper is divided into two parts with deliberately different prerequisites.  Part~\ref{part:field} contains the elementary field-valued
reconstruction and can be read by a reader with only a rudimentary working knowledge of matroids and algebraic geometry.  
Section~\ref{sec:setup} defines $\cT_n$ and its realization
classes, while Section~\ref{sec:well-defined} introduces marked plane cubics
and the map $\beta_k$.  Section~\ref{sec:seed} constructs the seed pencil and
the zero-sum-grid form of Chasles completion.  Sections
\ref{sec:odd-propagation} and~\ref{sec:even-propagation} carry out the
propagation, Section~\ref{sec:irreducible-smooth} proves irreducibility and
smoothness at the marked points, and Section~\ref{sec:grouplaw} recovers the
group law.  Section~\ref{sec:field-valued-bijection-proof} completes the
proof for $n\ge12$, and Section~\ref{sec:small-levels} treats the special cases $n=10,11$. 

Part~\ref{part:scheme} establishes the scheme-theoretic correspondence, and here the algebraic geometry prerequisites are notably higher.
Section~\ref{sec:schemes} constructs the ordinary realization scheme of a
matroid, and Section~\ref{sec:canonical-Tn} identifies the elliptic
realization scheme with a normalized frame slice.  Section
\ref{sec:modular-morphism} defines $X_1(n)^\circ$ scheme-theoretically and
constructs the modular morphism.  Section~\ref{sec:normalization} states the
Artinian-point criterion and compares points of $\cR_n$ with normalized
configuration deformations.  Section~\ref{sec:artinian-reconstruction}
proves the relative seed-pencil, propagation, and group-law results and then
deduces the global isomorphism.
Section~\ref{sec:complement} discusses a connection with
Mazur's theorem on rational torsion points on elliptic curves.  

Finally, Appendix~\ref{sec:bands} constructs the canonical
band-scheme model for the realization space of a matroid $M$, compares it with the foundation of $M$, 
and identifies its tropical points with the reduced Dressian of $M$. This perspective will be utilized in the subsequent work \cite{BES}.

\subsection{Statement on AI usage}
\label{sec:AI-usage} 

The author acknowledges the use of AI during the preparation of this manuscript. In fact, it would be fair to characterize this entire project as a collaboration with ChatGPT 5.5 Pro.
I came up with the basic algebraic and combinatorial strategy behind \Cref{thm:intro-field} myself after meditating on the papers by Borisov--Roulleau \cite{BR} and Green--Tao \cite{GT}, but GPT-5.5 came up with the precise $3\times 3$ zero-sum grid formalism which powers the combinatorial propagation argument as described below. Inspired by Borisov and Roulleau's proof that the complex realization space of $\cT_n$ is smooth, I also had the idea to enhance \Cref{thm:intro-field} to a scheme-theoretic isomorphism over $\Z[1/n]$ by working with Artinian deformations, but the technical scheme-theoretic arguments in Part 2 were worked out almost entirely by GPT-5.5.
Portions of this manuscript, including the majority of Part 2, were initially drafted by GPT-5.5 and then proofread, revised, and checked line-by-line by the author. Claude Fable 5 and GPT-5.6 Pro provided additional proofreading, as well as helpful references for the scheme-theoretic arguments and background material on modular curves.
The author has reviewed and takes full responsibility for all content in this paper.

\part{The field-valued correspondence}\label{part:field}

\section{The elliptic matroid and its realization space}\label{sec:setup}

In this section we define the family of matroids $\cT_n$ and the corresponding realization spaces $R_{\cT_n}(k)$.

\subsection{The elliptic matroid \texorpdfstring{$\cT_n$}{Tn}}

Fix an integer $n\ge 4$.

\begin{definition}
The \emph{elliptic matroid} $\cT_n$ is the rank-$3$ matroid on the ground set $\Z/n\Z = \{ 0,1,\ldots, n-1 \}$ whose non-bases of size 3 are the $3$-element subsets
$\{a,b,c\}\subset \Z/n\Z$ such that
\[
a+b+c = 0 \text{ in } \Z/n\Z.
\]
\end{definition}

The definition does indeed give a matroid: two distinct zero-sum triples cannot
have a pair of elements in common, and hence the zero-sum triples are the
circuit-hyperplanes of a rank-$3$ sparse paving matroid.

\subsection{Realizations and rescaling classes}

Let $M$ be a simple rank-$3$ matroid on a finite ground set $E$, and let $k$ be a field.

\begin{definition}
A \emph{realization} of $M$ over $k$ is a collection of nonzero vectors
\[
(v_e)_{e\in E}\in (k^3\setminus\{0\})^E
\]
such that, for every $3$-element subset $\{e_1,e_2,e_3\}\subset E$, the vectors $v_{e_1},v_{e_2},v_{e_3}$
are linearly dependent if and only if $\{e_1,e_2,e_3\}$ is a non-basis of $M$.
\end{definition}

\begin{definition}
Two realizations
\[
(v_e)_{e\in E}
\qquad\text{and}\qquad
(v'_e)_{e\in E}
\]
of $M$ over $k$ are \emph{rescaling-equivalent} if there exist
\[
g\in \operatorname{GL}_3(k)
\qquad\text{and}\qquad
(\lambda_e)_{e\in E}\in (k^\times)^E
\]
such that
\[
v'_e=\lambda_e\, g(v_e)
\qquad\text{for all } e\in E.
\]
We write $R_M(k)$ for the set of rescaling-equivalence classes of realizations of $M$ over $k$.
\end{definition}

Equivalently, a realization of $M$ over $k$ may be represented by a
$3\times|E|$ matrix with entries in $k$; rescaling equivalence is generated by left
multiplication by $\GL_3(k)$ and independent rescaling of its columns.
We use the terms \emph{realization} and \emph{representation} synonymously throughout the paper.

\begin{remark}
Geometrically, a realization of $\cT_n$ over $k$ is the same thing as a collection of points
\[
(P_i)_{i\in \Z/n\Z}\subset \PP^2(k)
\]
such that three distinct points $P_a,P_b,P_c$ are collinear if and only if
\[
a+b+c=0 \qquad\text{in }\Z/n\Z.
\]
The set $R_{\cT_n}(k)$ records such configurations up to projective linear automorphisms of $\PP^2$ and independent rescaling of the chosen homogeneous coordinates.
\end{remark}

\begin{lemma}\label{lem:marked-points-distinct}
Let $n\ge 4$, and let
\[
(P_i)_{i\in \Z/n\Z}\subset \PP^2(k)
\]
be a realization of $\cT_n$. Then the points $P_i$ are pairwise distinct.
\end{lemma}

\begin{proof}
Suppose $P_a=P_b$ with $a\ne b$. Choose
$c\in \Z/n\Z$ distinct from $a$, $b$, and $-a-b$; this is possible because $n\ge 4$. In homogeneous coordinates, the two columns representing $P_a$ and $P_b$ are proportional, so the three columns indexed by $a,b,c$ are linearly dependent. Hence $\{a,b,c\}$ is a non-basis of $\cT_n$, so $a+b+c=0$, contradicting the choice of $c$.
\end{proof}

\section{Marked plane cubics and the map \texorpdfstring{$\beta$}{beta}}
\label{sec:well-defined}

We now define the source of the map $\beta$ concretely in terms of plane cubics.

\subsection{Smooth and nodal marked cubics}

Let $k$ be a field, and let $E\subset \PP^2_k$ be an irreducible plane cubic.

If $E$ is smooth, then $E(k)$ carries the usual elliptic curve group law after choosing an identity element. 
If $E$ is nodal, then $E^\sm$ is a one-dimensional algebraic
torus over $k$ once an identity point is chosen; it becomes
isomorphic to $\G_m$ after base change to an algebraic
closure (see e.g. \cite{ART}).  We write the group law additively in both cases.

We first define the relevant triples; $X_1(n)^\circ(k)$ will then be the set of equivalence classes of such triples under the isomorphisms of Definition~\ref{def:marked-cubic-isomorphism}.

\begin{definition}
Let $n\ge 3$. A \emph{marked plane cubic of level $n$ over $k$} is a triple $(E,O,P)$ such that:
\begin{enumerate}[label=(\roman*)]
\item $E\subset \PP^2_k$ is an irreducible plane cubic which is either smooth or nodal;
\item $O\in E^\sm(k)$ satisfies
\[
\mathcal O_E(1)\cong \mathcal O_E(3O);
\]
\item $P\in E^\sm(k)$ has exact order $n$ in the group law on $E^\sm$ with identity $O$.
\end{enumerate}
\end{definition}

\begin{remark}
If $E$ is smooth, the condition $\mathcal O_E(1)\cong \mathcal O_E(3O)$
is equivalent in every characteristic to saying that $O$ is a flex of the given plane embedding, meaning that the tangent line at $O$ meets $E$ with multiplicity $3$.  
We use the line-bundle formulation because it also makes sense in the nodal case.
\end{remark}

\begin{definition}\label{def:marked-cubic-isomorphism}
An \emph{isomorphism}
\[
(E,O,P)\xrightarrow{\sim}(E',O',P')
\]
of marked plane cubics is a projective linear automorphism
$\varphi\in \operatorname{PGL}_3(k)$ such that
\[
\varphi(E)=E',\qquad \varphi(O)=O',\qquad \varphi(P)=P'.
\]
We define $X_1(n)^\circ(k)$ to be the set of isomorphism classes of marked plane cubics of level $n$ over $k$.
\end{definition}

\begin{remark}\label{rem:X1-open}
The notation $X_1(n)^\circ(k)$ reflects a modular interpretation that will be
made scheme-theoretic in Part~\ref{part:scheme}.  There we define an open
subscheme $X_1(n)^\circ\subset X_1(n)$ whose geometric fibers are smooth
elliptic curves and irreducible nodal generalized elliptic curves, and we show
that its field-valued points agree with the concrete projective isomorphism
classes defined above.
\end{remark}

\subsection{The realization attached to a marked cubic}

Let $(E,O,P)\in X_1(n)^\circ(k)$. For each $i\in \Z/n\Z$, define
\[
P_i:=iP\in E^\sm(k)\subset \PP^2(k).
\]
Choose nonzero homogeneous coordinates for each $P_i$, and let $A(E,O,P)$ be the resulting $3\times n$ matrix.
The following proposition shows that this produces a well-defined element of $R_{\cT_n}(k)$.

\begin{proposition}\label{prop:beta-welldefined}
Let $(E,O,P)\in X_1(n)^\circ(k)$. Then the points $(P_i)_{i\in \Z/n\Z}$ realize the matroid $\cT_n$; equivalently, for distinct $a,b,c\in \Z/n\Z$ one has
\[
P_a,\ P_b,\ P_c \text{ collinear}
\qquad\Longleftrightarrow\qquad
a+b+c=0.
\]
In particular, the rescaling class of $A(E,O,P)$ gives a well-defined element
\[
\beta(E,O,P)\in R_{\cT_n}(k).
\]
\end{proposition}

\begin{proof}
Let $H$ denote the divisor on $E$ cut out by a line in $\PP^2$. By assumption,
\[
\mathcal O_E(H)\cong \mathcal O_E(3O).
\]

Since $E$ is a plane cubic, the restriction map
\[
H^0(\PP^2,\mathcal O(1))\longrightarrow H^0(E,\mathcal O_E(1))
\]
is an isomorphism and the given plane embedding is the complete linear
system associated to $\mathcal O_E(1)\cong\mathcal O_E(3O)$.

Let $a,b,c\in \Z/n\Z$ be distinct, and suppose first that $a+b+c=0$. Then in the group law on $E^\sm$ with identity $O$ we have
\[
P_a+P_b+P_c=(a+b+c)P=0.
\]
For an irreducible plane cubic embedded by the complete linear system $|3O|$, three distinct smooth points are collinear if and only if their sum is zero in the group law. Hence $P_a,P_b,P_c$ are collinear.

Conversely, suppose $P_a,P_b,P_c$ are collinear. Then the line through them cuts out the divisor
$P_a+P_b+P_c$ on $E$, and this divisor is linearly equivalent to $H$. Since $H\sim 3O$, it follows that
$P_a+P_b+P_c\sim 3O$, or equivalently, $P_a+P_b+P_c=0$ in the group law on $E^\sm$. Since $P_i=iP$, this gives $(a+b+c)P=0$.

As $P$ has exact order $n$, we conclude that
\[
a+b+c=0 \qquad\text{in }\Z/n\Z.
\]

This proves the equivalence, and hence shows that the rescaling class of $A(E,O,P)$ defines an element of $R_{\cT_n}(k)$. It is clearly independent of the choice of homogeneous coordinates for the $P_i$, and changing the projective coordinates on $\PP^2$ only changes the representing matrix by left multiplication by an element of $\operatorname{GL}_3(k)$. Thus $\beta$ is well-defined.
\end{proof}

\subsection{The field-valued correspondence}\label{sec:mainresults}

We now state the field-valued form of the correspondence.
For every field $k$, recall that $R_{\cT_n}(k)$ denotes the set of rescaling classes of
$k$-realizations of the elliptic matroid $\cT_n$.

\begin{theorem}[Field-valued Borisov--Roulleau correspondence]\label{thm:field-valued-bijection}
Let $n\ge 10$, and let $k$ be a field with
$\operatorname{char}(k)\nmid n$.
Then the map
\[
\beta_k:X_1(n)^\circ(k)\longrightarrow R_{\cT_n}(k)
\]
is bijective.
\end{theorem}

We first prove the easier half of Theorem~\ref{thm:field-valued-bijection}, the assertion that the map $\beta$ is injective:

\begin{proposition}[Injectivity in Theorem~\ref{thm:field-valued-bijection}]
\label{prop:field-valued-injection}
Let $n\ge 10$, and let $k$ be any field.  Then the map
\[
\beta_k:X_1(n)^\circ(k)\longrightarrow R_{\cT_n}(k)
\]
is injective.
\end{proposition}

\begin{proof}
Let $(E,O,P)$ and $(E',O',P')$ be two elements of $X_1(n)^\circ(k)$ with
\[
\beta(E,O,P)=\beta(E',O',P').
\]
We must show that the two triples are isomorphic.

Let
\[
Q_i=iP\in E^\sm(k),\qquad Q'_i=iP'\in (E')^\sm(k).
\]
Choose homogeneous lifts $v_i$ and $v'_i$ for $Q_i$ and $Q'_i$, respectively. Since $\beta(E,O,P)=\beta(E',O',P')$
in $R_{\cT_n}(k)$, the two associated realizations are rescaling-equivalent. Hence there exist
$g\in \operatorname{GL}_3(k)$ and scalars $\lambda_i\in k^\times$ such that
$v'_i=\lambda_i\,g(v_i)$ for all $i\in \Z/n\Z$.
Thus the projective automorphism $\varphi$ associated to $g$ satisfies
$\varphi(Q_i)=Q'_i$ for all $i$.
Replacing $(E',O',P')$ by the isomorphic marked plane cubic
$(\varphi^{-1}E',\varphi^{-1}O',\varphi^{-1}P')$,
we may therefore assume that
\[
Q_i=Q'_i
\qquad\text{for all }i\in \Z/n\Z.
\]

We claim first that $E=E'$ as plane cubics in $\PP^2_k$.  Indeed, both $E$ and $E'$ are irreducible cubics.  If they were distinct, then they would have no common irreducible component.  But they contain the same $n\ge 10$ marked points.  By B\'ezout's theorem, two distinct plane cubics without a common component meet in at most $9$ points, a contradiction.  Hence $E=E'$.

Since
\[
Q_0=0\cdot P=O
\qquad\text{and}\qquad
Q'_0=0\cdot P'=O',
\]
and $Q_0 = Q'_0$, we have $O = O'$.
Similarly,
\[
Q_1=1\cdot P=P
\qquad\text{and}\qquad
Q'_1=1\cdot P'=P',
\]
and since $Q_1 = Q'_1$ we have $P=P'$.
\end{proof}

The proof that $\beta$ is surjective will take a good bit of additional work, which we will undertake now.

\section{A seed pencil through nine points}\label{sec:seed}

Throughout this section and the next two, let $k$ be a field, and let
\[
(P_i)_{i\in \Z/n\Z}\subset \PP^2(k)
\]
be a realization of the elliptic matroid $\cT_n$ over $k$. Thus, for distinct indices $a,b,c\in \Z/n\Z$, one has
\[
P_a,\ P_b,\ P_c \text{ collinear}
\qquad\Longleftrightarrow\qquad
a+b+c=0 \quad\text{in }\Z/n\Z.
\]

We begin with a simple combinatorial observation.

\begin{lemma}\label{lem:no-four-collinear}
Let $n\ge 4$, let $k$ be a field, and let
\[
(P_i)_{i\in \Z/n\Z}\subset \PP^2(k)
\]
be a realization of the elliptic matroid $\cT_n$. Then no line in $\PP^2_k$ contains four distinct points of the configuration.
\end{lemma}

\begin{proof}
Suppose a line $\ell$ contained four distinct points
\[
P_a,\ P_b,\ P_c,\ P_d.
\]
Then both triples $\{a,b,c\}$ and $\{a,b,d\}$ would be dependent in the matroid $\cT_n$, so
\[
a+b+c=0
\qquad\text{and}\qquad
a+b+d=0
\]
in $\Z/n\Z$. Hence $c=d$, contradiction.
\end{proof}

We now construct a special pencil of cubics whose properties we will systematically exploit.

\begin{proposition}\label{prop:seed-pencil}
Suppose $n\ge 9$, and let
\[
L_1=\overline{P_0P_4P_{-4}},\qquad
L_2=\overline{P_1P_2P_{-3}},\qquad
L_3=\overline{P_{-1}P_{-2}P_3},
\]
and
\[
M_1=\overline{P_0P_2P_{-2}},\qquad
M_2=\overline{P_1P_3P_{-4}},\qquad
M_3=\overline{P_{-1}P_{-3}P_4}.
\]
Define reducible cubics
\[
D_1=L_1\cup L_2\cup L_3,
\qquad
D_2=M_1\cup M_2\cup M_3.
\]
Then:
\begin{enumerate}[label=\textup{(\roman*)}]
\item the six lines $L_1,L_2,L_3,M_1,M_2,M_3$ are pairwise distinct;
\item the scheme-theoretic intersection $D_1\cap D_2$ consists of the nine reduced points
\[
P_{-4},P_{-3},\dots,P_4;
\]
\item the space of cubic forms vanishing on these nine points is the pencil
$\langle D_1,D_2\rangle$.
\end{enumerate}
\end{proposition}

\begin{proof}
Each displayed triple sums to $0$ in $\Z/n\Z$, so the six lines are well-defined.

To prove \textup{(i)}, it suffices to note that if two of these six lines were equal, then that line would contain at least four distinct marked points, contradicting Lemma~\ref{lem:no-four-collinear}. For example, if $L_1=L_2$, then the common line contains
\[
P_0,\ P_4,\ P_{-4},\ P_1,\ P_2,\ P_{-3},
\]
which are six distinct points because $n\ge 9$.  In every other pair among the six displayed triples, the union of the two index triples has cardinality at least four in $\Z/n\Z$ for $n\ge 9$, so the same argument applies.

For \textup{(ii)}, consider the $3\times 3$ array of pairwise intersections:
\[
\begin{array}{c|ccc}
 & M_1 & M_2 & M_3 \\ \hline
L_1 & P_0 & P_{-4} & P_4 \\
L_2 & P_2 & P_1 & P_{-3} \\
L_3 & P_{-2} & P_3 & P_{-1}
\end{array}
\]

For each entry, the corresponding point lies on both the given row-line and the given column-line, and the displayed nine indices are pairwise distinct.  Since $D_1$ and $D_2$ have no common irreducible component, B\'ezout's theorem gives
\[
\operatorname{length}(D_1\cap D_2)=3\cdot3=9
\]
after base change to an algebraic closure.  The nine displayed points are distinct and each contributes intersection multiplicity at least $1$.  Hence each has multiplicity exactly $1$, and there are no further intersection points.  Thus
\[
D_1\cap D_2=\{P_{-4},P_{-3},\dots,P_4\}
\]
as a reduced scheme.

For \textup{(iii)}, let
\[
Z:=\{P_{-4},P_{-3},\dots,P_4\}\subset \PP^2_k.
\]
By \textup{(ii)}, the scheme $Z$ is the complete intersection of the two cubics $D_1$ and $D_2$. Hence its homogeneous ideal is
$I_Z=(D_1,D_2)$,
and therefore the degree-$3$ part of the ideal is 
$(I_Z)_3=\langle D_1,D_2\rangle$.
Equivalently, the cubics passing through the nine points of $Z$ form the pencil spanned by $D_1$ and $D_2$.
\end{proof}

\begin{corollary}\label{cor:unique-window-cubic}
If $n \ge 10$, there is a unique cubic $C\subset \PP^2_k$
passing through the ten points
\[
P_{-4},P_{-3},\dots,P_5.
\]
\end{corollary}

\begin{proof}
By Proposition~\ref{prop:seed-pencil}, the cubics through
$P_{-4},P_{-3},\dots,P_4$ form the $2$-dimensional vector space $\langle D_1,D_2\rangle$.
Since
\[
D_1\cap D_2=\{P_{-4},P_{-3},\dots,P_4\}
\]
scheme-theoretically and $n \ge 10$, the point $P_5$ does not lie on both $D_1$ and $D_2$. Equivalently, evaluation at $P_5$ defines a nonzero linear functional on the vector space $\langle D_1,D_2\rangle$. Its kernel is therefore $1$-dimensional, and this kernel corresponds to a unique cubic (up to scalar multiples) passing through the ten points $P_{-4},P_{-3},\dots,P_5$.
\end{proof}

The rest of the argument is based on a family of $3\times 3$ ``zero-sum grids.''
We will make use of the classical Chasles theorem\footnote{This result is due to Chasles; it is Theorem~CB3 of \cite{EGH}, where
the history is traced in detail.  It is commonly called the Cayley--Bacharach
theorem, but that result concerns the extension of Chasles' theorem to curves of higher degree.
As discussed in \cite{EGH}, Cayley's 1843 extension to intersections of curves of higher
degree was, as stated, false, and the missing term was supplied by Bacharach in
1886.} (see, for example, \cite{EGH}):

\begin{theorem}[Chasles]\label{thm:CBC}
Let $C_1,C_2\subset \PP^2_k$ be cubic curves with no common irreducible component. Suppose that
\[
C_1\cap C_2=\{Q_1,\dots,Q_9\}
\]
consists of nine reduced $k$-rational points. Then any cubic curve passing through eight of the points $Q_1,\dots,Q_9$ necessarily passes through the ninth.
\end{theorem}

\begin{remark}
We will only apply Theorem~\ref{thm:CBC} in the case where $C_1$ and $C_2$ are unions of three distinct lines.  The stated form of the theorem over an arbitrary field follows from the usual algebraically closed version by base change: a common geometric component would descend, and the nine reduced $k$-rational intersection points remain reduced after scalar extension.
\end{remark}

\begin{definition}\label{def:gamma-grid}
For elements $a,b,c,d\in \Z/n\Z$, define the associated \emph{zero-sum grid}
\[
\Gamma(a,b,c,d):=
\begin{pmatrix}
a & b & -a-b \\
c & d & -c-d \\
-a-c & -b-d & a+b+c+d
\end{pmatrix}.
\]
\end{definition}

Every row sum and every column sum of $\Gamma(a,b,c,d)$ is equal to $0$ in $\Z/n\Z$. Thus, whenever the nine entries are pairwise distinct, each row and each column of the grid gives a collinear triple of marked points.

\begin{lemma}[Chasles completion for a zero-sum grid]\label{lem:grid-completion}
Let $X\subset \PP^2_k$ be a cubic curve. Suppose that the nine entries of $\Gamma(a,b,c,d)$
are pairwise distinct in $\Z/n\Z$. If $X$ contains the eight marked points corresponding to any eight of these nine entries, then $X$ contains the ninth marked point as well.
\end{lemma}

\begin{proof}
Let
\[
\Gamma(a,b,c,d)=
\begin{pmatrix}
u_{11} & u_{12} & u_{13} \\
u_{21} & u_{22} & u_{23} \\
u_{31} & u_{32} & u_{33}
\end{pmatrix}.
\]
For $i=1,2,3$, let $R_i$ be the line through the three points in row $i$, and for $j=1,2,3$, let $C_j$ be the line through the three points in column $j$.

Because the nine entries are pairwise distinct, each $R_i$ contains three distinct marked points, and each $C_j$ contains three distinct marked points. By Lemma~\ref{lem:no-four-collinear}, the three row-lines are pairwise distinct, the three column-lines are pairwise distinct, and no row-line can coincide with a column-line: indeed, if $R_i=C_j$, then that common line would contain the five distinct grid points in row $i$ and column $j$.

Now set
\[
R:=R_1\cup R_2\cup R_3,
\qquad
C:=C_1\cup C_2\cup C_3.
\]
These are cubics with no common irreducible component, and the nine grid points are distinct points of $R\cap C$.  By B\'ezout's theorem, the total intersection length is $9$ after base change to an algebraic closure.  It follows that the nine grid points are the entire intersection, each with multiplicity $1$.  Thus $R\cap C$ is a reduced zero-dimensional scheme of degree $9$, and Theorem~\ref{thm:CBC} shows that any cubic through eight of the nine grid points must also pass through the ninth.
\end{proof}

\section{Propagation in the odd case}\label{sec:odd-propagation}

In this section we assume that
\[
n=2m+1\ge 13
\]
is odd. We identify $\Z/n\Z$ with the symmetric set of representatives
\[
\{-m,-m+1,\dots,-1,0,1,\dots,m\}.
\]
For integers $r\le s$ in this range, we write
\[
[r,s]:=\{r,r+1,\dots,s\}\subset \Z/n\Z.
\]

Let $C\subset \PP^2_k$ be the unique cubic of Corollary~\ref{cor:unique-window-cubic}. Thus
\[
P_i\in C \qquad\text{for all } i\in [-4,5].
\]

Our first goal is to enlarge this to the symmetric block $[-5,5]$.

\begin{lemma}\label{lem:odd-first-bootstrap}
The cubic $C$ contains the point $P_{-5}$. Equivalently,
\[
[-5,5]\subseteq \{\,i\in \Z/n\Z : P_i\in C\,\}.
\]
\end{lemma}

\begin{proof}
Consider the grid
\[
\Gamma(-5,0,1,3)=
\begin{pmatrix}
-5 & 0 & 5 \\
1 & 3 & -4 \\
4 & -3 & -1
\end{pmatrix}.
\]
Its nine entries are distinct integers in the interval $[-5,5]$, hence are distinct in $\Z/n\Z$ because $n\ge 11$.
The eight entries
\[
0,\ 5,\ 1,\ 3,\ -4,\ 4,\ -3,\ -1
\]
all lie in the window $[-4,5]$, so the corresponding eight marked points lie on $C$ by construction. By Lemma~\ref{lem:grid-completion}, the ninth point $P_{-5}$ also lies on $C$.
\end{proof}

We now establish the basic propagation step.

\begin{lemma}\label{lem:odd-propagation-step}
Let $r$ be an integer with $5\le r<m$.
Suppose that
\[
P_i\in C \qquad\text{for all } i\in [-r,r].
\]
Then
\[
P_{r+1}\in C
\qquad\text{and}\qquad
P_{-(r+1)}\in C.
\]
In other words,
\[
[-r,r]\subseteq \{\,i : P_i\in C\,\}
\qquad\Longrightarrow\qquad
[-(r+1),r+1]\subseteq \{\,i : P_i\in C\,\}.
\]
\end{lemma}

\begin{proof}
First consider the grid
\[
\Gamma(0,1,2,r-2)=
\begin{pmatrix}
0 & 1 & -1 \\
2 & r-2 & -r \\
-2 & 1-r & r+1
\end{pmatrix}.
\]
Since $r\ge 5$, the entries of this grid are pairwise distinct. Moreover, they all lie in the interval $[-r,r+1]$, whose length is
\[
(r+1)-(-r)=2r+1<n=2m+1
\]
because $r<m$. Therefore the entries are also pairwise distinct in $\Z/n\Z$.

All entries except $r+1$ lie in $[-r,r]$, so the corresponding marked points lie on $C$ by hypothesis. Lemma~\ref{lem:grid-completion} therefore implies that
\[
P_{r+1}\in C.
\]

Now apply the same argument to the negated grid
\[
-\Gamma(0,1,2,r-2)=
\begin{pmatrix}
0 & -1 & 1 \\
-2 & 2-r & r \\
2 & r-1 & -(r+1)
\end{pmatrix}.
\]
Again the entries are pairwise distinct in $\Z/n\Z$, and all entries except $-(r+1)$ lie in $[-r,r]$. 
Lemma~\ref{lem:grid-completion} gives
\[
P_{-(r+1)}\in C,
\]
completing the argument.
\end{proof}

We can now propagate to the whole configuration.

\begin{proposition}\label{prop:odd-all-on-cubic}
Assume that $n=2m+1\ge 13$ is odd. Then the cubic $C$ of Corollary~\ref{cor:unique-window-cubic} contains every marked point $P_i$.
\end{proposition}

\begin{proof}
By Lemma~\ref{lem:odd-first-bootstrap}, the cubic $C$ contains the full block $[-5,5]$.
By hypothesis, $m\ge 6$, and Lemma~\ref{lem:odd-propagation-step} applies successively for $r=5,6,\dots,m-1$.
Starting from $[-5,5]$, we obtain
\[
[-6,6],\ [-7,7],\ \dots,\ [-m,m].
\]
But $[-m,m]$ is the whole of $\Z/n\Z$ in our chosen system of representatives. Therefore
\[
P_i\in C \qquad\text{for all } i\in \Z/n\Z.
\]
\end{proof}

\section{Propagation in the even case}\label{sec:even-propagation}

In this section we assume that
\[
n=2m\ge 12
\]
is even. We identify $\Z/n\Z$ with the set of representatives
\[
\{-(m-1),-(m-2),\dots,-1,0,1,\dots,m\}.
\]
Thus the only residue class not represented symmetrically is the half-period class $m=-m$.

Let $C\subset \PP^2_k$ again denote the unique cubic of Corollary~\ref{cor:unique-window-cubic}. We first prove the same initial bootstrap as in the odd case.

\begin{lemma}\label{lem:even-first-bootstrap}
The cubic $C$ contains the point $P_{-5}$. Equivalently,
\[
[-5,5]\subseteq \{\,i\in \Z/n\Z : P_i\in C\,\}.
\]
\end{lemma}

\begin{proof}
The same grid as before works:
\[
\Gamma(-5,0,1,3)=
\begin{pmatrix}
-5 & 0 & 5 \\
1 & 3 & -4 \\
4 & -3 & -1
\end{pmatrix}.
\]
Since $n\ge 12$, the entries of this grid are pairwise distinct in $\Z/n\Z$
Eight of the nine entries lie in the window $[-4,5]$, so the corresponding eight marked points lie on $C$. Lemma~\ref{lem:grid-completion} forces the remaining point $P_{-5}$ to lie on $C$ as well.
\end{proof}

The next lemma is the same propagation step as in the odd case.

\begin{lemma}\label{lem:even-propagation-step}
Let $r$ be an integer with
\[
5\le r\le m-2.
\]
Suppose that
\[
P_i\in C \qquad\text{for all } i\in [-r,r].
\]
Then
\[
P_{r+1}\in C
\qquad\text{and}\qquad
P_{-(r+1)}\in C.
\]
\end{lemma}

\begin{proof}
Exactly the same grids as in Lemma~\ref{lem:odd-propagation-step} apply:
\[
\Gamma(0,1,2,r-2)=
\begin{pmatrix}
0 & 1 & -1 \\
2 & r-2 & -r \\
-2 & 1-r & r+1
\end{pmatrix}
\]
and its negative. The entries are pairwise distinct integers lying in $[-r,r+1]$. Since
\[
(r+1)-(-r)=2r+1\le 2(m-2)+1 = 2m-3 < 2m=n,
\]
they remain pairwise distinct modulo $n$. All entries except the desired new one lie in the known block $[-r,r]$, so Lemma~\ref{lem:grid-completion} gives both
\[
P_{r+1}\in C
\qquad\text{and}\qquad
P_{-(r+1)}\in C.
\]
\end{proof}

Starting from $[-5,5]$, this propagates up to the block $[-(m-1),m-1]$. One residue class remains: the half-period $m$.

\begin{lemma}\label{lem:half-period-grid}
Assume that
\[
P_i\in C \qquad\text{for all } i\in [-(m-1),m-1].
\]
Then
\[
P_m\in C.
\]
\end{lemma}

\begin{proof}
Consider the grid
\[
\Gamma_m :=
\begin{pmatrix}
0 & 1 & -1 \\
2 & m-3 & -(m-1) \\
-2 & -(m-2) & m
\end{pmatrix}.
\]
Each row sum and each column sum is $0$ in $\Z/2m\Z$, so this is indeed a zero-sum grid.

Since $m\ge 6$, the entries of $\Gamma_m$ are pairwise distinct as integers. Moreover, they all lie in the interval $[-(m-1),m]$, whose length is
\[
m-(-(m-1)) = 2m-1 < 2m=n.
\]
Hence they are pairwise distinct in $\Z/n\Z$.
All entries except $m$ belong to the known block $[-(m-1),m-1]$. Therefore Lemma~\ref{lem:grid-completion} implies that $P_m\in C$ as well.
\end{proof}

We may now conclude the even case.

\begin{proposition}\label{prop:even-all-on-cubic}
Assume that $n=2m\ge 12$ is even. Then the cubic $C$ of Corollary~\ref{cor:unique-window-cubic} contains every marked point $P_i$.
\end{proposition}

\begin{proof}
By Lemma~\ref{lem:even-first-bootstrap}, the cubic $C$ contains $[-5,5]$.

If $m=6$, then this already gives the block $[-(m-1),m-1]=[-5,5]$, so Lemma~\ref{lem:half-period-grid} immediately yields $P_m\in C$.
If $m\ge 7$, apply Lemma~\ref{lem:even-propagation-step} successively for
\[
r=5,6,\dots,m-2.
\]
This shows that
\[
P_i\in C \qquad\text{for all } i\in [-(m-1),m-1]
\]
and Lemma~\ref{lem:half-period-grid} gives $P_m\in C$.
Thus $C$ contains every residue class in our chosen system of representatives for $\Z/n\Z$, hence
\[
P_i\in C \qquad\text{for all } i\in \Z/n\Z.
\]
\end{proof}

Combining the odd and even cases, we obtain the main conclusion of this part of the argument.

\begin{theorem}\label{thm:all-points-on-cubic}
Let $n\ge 12$, let $k$ be any field, and let
\[
(P_i)_{i\in \Z/n\Z}\subset \PP^2(k)
\]
be a realization of $\cT_n$. Then there exists a cubic $C\subset \PP^2_k$
containing all the points $P_i$.
More precisely, $C$ is the unique cubic passing through the ten-point window
\[
P_{-4},P_{-3},\dots,P_5.
\]
\end{theorem}

\begin{proof}
The existence and uniqueness of a cubic through the ten-point window is Corollary~\ref{cor:unique-window-cubic}. Proposition~\ref{prop:odd-all-on-cubic} shows that this cubic contains all marked points when $n$ is odd, and Proposition~\ref{prop:even-all-on-cubic} shows the same when $n$ is even.
\end{proof}

\section{Irreducibility and smoothness}\label{sec:irreducible-smooth}

In this section we assume that $n \ge 12$. We continue to assume that $k$ is a field, and that
\[
(P_i)_{i\in \Z/n\Z}\subset \PP^2(k)
\]
is a realization of $\cT_n$. Let $C\subset \PP^2_k$ be the cubic provided by Theorem~\ref{thm:all-points-on-cubic}, so that
\[
P_i\in C \qquad\text{for all } i\in \Z/n\Z.
\]

Our goal is to prove that $C$ is irreducible and that every marked point $P_i$ lies in the smooth locus $C^\sm$.

\subsection{A combinatorial lemma}

\begin{lemma}\label{lem:zero-sum-triple-outside-small-set}
Let $n\ge 12$, and let $S\subset \Z/n\Z$ be a subset with $|S|\le 3$.
Then there exist three distinct elements
$a,b,c\in (\Z/n\Z) \setminus S$ such that $a+b+c=0$.
\end{lemma}

\begin{proof}
The number of ordered triples $(a,b,c)\in(\Z/n\Z)^3$ with $a+b+c=0$
is $n^2$.  Among these, the triples with at least two equal entries are those
with $a=b$, or $a=c$, or $b=c$.  Each of these three conditions gives $n$
ordered triples.  Their pairwise intersections coincide with the set of triples
$a=b=c$, which has $\gcd(n,3)$ elements.  Hence the number of ordered
zero-sum triples with $a,b,c$ pairwise distinct is
\[
        n^2-3n+2\gcd(n,3).
\]
Dividing by $6$, the number of unordered zero-sum triples of distinct elements is
\[
        \frac{n^2-3n+2\gcd(n,3)}{6}.
\]

Now fix an element $s\in\Z/n\Z$.  The number of unordered zero-sum triples of distinct elements containing $s$ is at most $n/2$: indeed, once $s$ is fixed, the other two elements have the form $b$ and $-s-b$, and the unordered pair $\{b,-s-b\}$ is determined up to interchanging its two entries.
Therefore the number of unordered zero-sum triples meeting $S$ is at most $3n/2$.
For $n\ge 12$, one easily checks that
\[
        \frac{n^2-3n+2\gcd(n,3)}{6}>\frac{3n}{2}.
\]
(Indeed, this is equivalent to
\[
         n^2-12n+2\gcd(n,3)>0,
\]
which is clear for $n\ge 13$, and for $n=12$ gives $144-144+6>0$.)
Hence there is a zero-sum triple disjoint from $S$, as desired.
\end{proof}

\subsection{Irreducibility}

\begin{proposition}\label{prop:irreducible}
The cubic $C$ is irreducible.
\end{proposition}

\begin{proof}
We first rule out the possibility that $C$ is supported on a union of three lines.
By Lemma~\ref{lem:no-four-collinear}, each line contains at most three distinct marked points. Hence a union of three lines can contain at most nine marked points. Since $C$ contains all $n\ge 12$ points $P_i$, this is impossible.
Thus, if $C$ were reducible, it would have to be of the form $C=Q\cup L$,
where $L$ is a line and $Q$ is a conic. Since the case of three lines has already been excluded, the conic $Q$ must be irreducible.

Let
\[
S:=\{\,i\in \Z/n\Z : P_i\in L\,\}.
\]
By Lemma~\ref{lem:no-four-collinear}, the line $L$ contains at most three marked points, so $|S|\le 3$.
Applying Lemma~\ref{lem:zero-sum-triple-outside-small-set}, we obtain distinct
$a,b,c\notin S$ such that $a+b+c=0$.
The corresponding three points $P_a,P_b,P_c$ all lie on the conic $Q$.  Since their indices sum to zero, these three points are collinear in the realization of $\cT_n$.
But then there is a line which meets the irreducible conic $Q$ in three distinct points, which is impossible by B\'ezout's theorem. This contradiction proves that $C$ is irreducible.
\end{proof}

\subsection{Smoothness of the marked points}

\begin{lemma}\label{lem:through-each-point-a-triple}
Let $n\ge 5$, let $k$ be a field, and let
\[
(P_i)_{i\in \Z/n\Z}\subset \PP^2(k)
\]
be a realization of the elliptic matroid $\cT_n$. Then for each $i\in \Z/n\Z$, there exists $j \in \Z/n\Z$ such that the elements $i,j,-i-j$
of $\Z/n\Z$ are distinct.
Equivalently, every marked point $P_i$ lies on a line containing two other distinct marked points of the configuration.
\end{lemma}

\begin{proof}
Fix $i\in \Z/n\Z$. We seek $j\in \Z/n\Z$ such that the three elements $i,j,-i-j$ are pairwise distinct.
The bad values of $j$ are exactly those for which one of the following holds:
\[
j=i,\qquad -i-j=i,\qquad -i-j=j.
\]
These conditions are equivalent to
\[
j=i,\qquad j=-2i,\qquad 2j=-i.
\]

The first two each exclude one value of $j$. The equation $2j=-i$
has at most two solutions in $\Z/n\Z$. Hence altogether there are at most four bad values of $j$.

Since $n\ge 5$, there exists some $j\in \Z/n\Z$ avoiding all these bad values. For this choice of $j$, the three elements $i,j,-i-j$ 
are pairwise distinct and sum to zero. Therefore the corresponding three points are collinear in the realization of $\cT_n$.
\end{proof}

\begin{proposition}\label{prop:smooth-points}
Every marked point $P_i$ lies in the smooth locus $C^\sm$.
\end{proposition}

\begin{proof}
Fix $i\in \Z/n\Z$. By Lemma~\ref{lem:through-each-point-a-triple}, there exist distinct elements $i,j,-i-j$ 
such that the three corresponding points $P_i$, $P_j$, and $P_{-i-j}$ are collinear.

Suppose for the sake of contradiction that $P_i$ is singular on the irreducible cubic $C$. Then every line through $P_i$ meets $C$ with intersection multiplicity at least $2$ at $P_i$. 
Since a line has total intersection number $3$ with a cubic, such a line can meet $C$ in at most one additional point.
But the line through $P_i$, $P_j$, and $P_{-i-j}$ contains two further distinct points of $C$, a contradiction.

Therefore $P_i$ is smooth on $C$. Since $i$ was arbitrary, every marked point lies in $C^\sm$.
\end{proof}

Combining the results of this section with Theorem~\ref{thm:all-points-on-cubic}, we obtain:

\begin{theorem}\label{thm:cubic-irreducible-smooth}
Let $n\ge 12$, let $k$ be a field, and let
\[
(P_i)_{i\in \Z/n\Z}\subset \PP^2(k)
\]
be a realization of $\cT_n$. Then there exists an irreducible cubic $C\subset \PP^2_k$
such that
\[
P_i\in C^\sm(k)
\qquad\text{for all } i\in \Z/n\Z.
\]
More precisely, $C$ is the unique cubic passing through the ten-point window
\[
P_{-4},P_{-3},\dots,P_5.
\]
\end{theorem}

\begin{proof}
The existence and uniqueness of a cubic containing all marked points is Theorem~\ref{thm:all-points-on-cubic}. Proposition~\ref{prop:irreducible} shows that this cubic is irreducible, and Proposition~\ref{prop:smooth-points} shows that every marked point is smooth.
\end{proof}

Since the cubic $C$ whose existence is guaranteed by
Theorem~\ref{thm:cubic-irreducible-smooth} is irreducible and has a
smooth $k$-rational point, it is geometrically integral by
\cite[Tag~0CDW]{Stacks}, and hence geometrically irreducible.
Thus its smooth locus carries the usual algebraic group structure once
a smooth identity point is chosen.

\section{Recovering the group law}\label{sec:grouplaw}

We retain the hypotheses of Theorem~\ref{thm:cubic-irreducible-smooth}. Thus $C$ is an irreducible plane cubic, all marked points lie in $C^\sm(k)$, and the collinearity relations among the marked points are exactly those prescribed by the matroid $\cT_n$.

We now endow the smooth locus $C^\sm$ with a group law whose identity element is $P_0$, and show that the present plane embedding already realizes $P_0$ as the distinguished point cut out by the line bundle $\mathcal O_C(1)\cong \mathcal O_C(3P_0)$. In particular, the given configuration itself is of the standard form $P_i=iP_1$.

\subsection{The line-section constant}

Let $G:=C^\sm$ be the smooth locus of $C$, equipped with the algebraic group structure whose identity element is $P_0$. We write the group law additively. For each $i\in \Z/n\Z$, let
$x_i\in G(k)$ denote the group element corresponding to the point $P_i$, so that $x_0 = 0$. 

Let $H$ be the divisor cut out on $C$ by a line in $\PP^2$. Under the standard identification \cite{ART} of the smooth locus with the generalized Jacobian,
\[
G(k)\cong \Pic^0(C)(k),\qquad Q\longmapsto [Q-P_0],
\]
define
\[
\lambda := [H-3P_0]\in G(k).
\]

\begin{lemma}\label{lem:line-section-constant}
Let $Q,R,S\in C^\sm(k)$ be three distinct collinear points. Then
\[
Q+R+S=\lambda
\]
in the group law on $G$.
\end{lemma}

\begin{proof}
The line through $Q,R,S$ cuts out the divisor $Q+R+S$
on $C$, and this divisor is linearly equivalent to $H$. Therefore
\[
[Q-P_0]+[R-P_0]+[S-P_0]
=
[H-3P_0]
=
\lambda
\]
in $\Pic^0(C)(k)$. Under the identification with $G(k)$, this is exactly the claimed identity.
\end{proof}

In particular, whenever $a,b,c\in \Z/n\Z$ are distinct and satisfy $a+b+c=0$,
the corresponding marked points are collinear, and hence $x_a+x_b+x_c=\lambda$.

\subsection{A window computation}

\begin{lemma}[Window computation]\label{lem:window-computation}
Let $C\subset \PP^2_k$ be an irreducible plane cubic, and suppose that
$P_{-5},P_{-4},\ldots,P_5$ are pairwise distinct points of $C^\sm(k)$.  Equip
$G=C^\sm$ with the group law having identity $P_0$, and let
$x_i\in G(k)$ denote the element corresponding to $P_i$.  Let $H$
be the divisor cut out on $C$ by a line in $\PP^2$, and define $\lambda:=[H-3P_0]\in G(k)$.
Assume that the triples summing to zero as integers among these points are collinear.  Then $\lambda=0$,
$x_i=i\cdot x_1$ for every $i\in \{-5,-4,-3,-2,-1,0,1,2,3,4,5\}$, and
$\mathcal O_C(1)\cong \mathcal O_C(3P_0)$.
\end{lemma}

\begin{proof}
We use only collinear triples among the points in the window $[-5,5]$.
From the dependent triples
\[
(0,1,-1),\qquad (0,2,-2),\qquad (-1,-2,3),\qquad (0,3,-3),\qquad (1,2,-3),
\]
Lemma~\ref{lem:line-section-constant} gives
\begin{align}
x_1+x_{-1} &= \lambda, \label{eq:local1}\\
x_2+x_{-2} &= \lambda, \label{eq:local2}\\
x_{-1}+x_{-2}+x_3 &= \lambda, \label{eq:local3}\\
x_3+x_{-3} &= \lambda, \label{eq:local4}\\
x_1+x_2+x_{-3} &= \lambda. \label{eq:local5}
\end{align}
From \eqref{eq:local1} and \eqref{eq:local2} we obtain
\[
x_{-1}=\lambda-x_1,
\qquad
x_{-2}=\lambda-x_2.
\]
Substituting these into \eqref{eq:local3} yields
\[
x_3=\lambda-(\lambda-x_1)-(\lambda-x_2)=x_1+x_2-\lambda.
\]
Then \eqref{eq:local4} gives
\[
x_{-3}=\lambda-x_3=2\lambda-x_1-x_2.
\]
Finally, substituting this into \eqref{eq:local5}, we get
\[
x_1+x_2+(2\lambda-x_1-x_2)=\lambda,
\]
hence $\lambda=0$.  Thus $[H-3P_0]=0$, which is equivalent to
\[
\mathcal O_C(1)\cong \mathcal O_C(3P_0).
\]

Now \eqref{eq:local1} and \eqref{eq:local2} reduce to
\[
x_{-1}=-x_1,
\qquad
x_{-2}=-x_2.
\]
Equation \eqref{eq:local3} becomes $x_3=x_1+x_2$, and then \eqref{eq:local4}
gives
\[
x_{-3}=-x_1-x_2.
\]
Next use the dependent triple $(-1,-3,4)$.  Since $\lambda=0$,
Lemma~\ref{lem:line-section-constant} gives
\[
x_{-1}+x_{-3}+x_4=0.
\]
Substituting the known values, we obtain
\[
-x_1+(-x_1-x_2)+x_4=0,
\]
so
\[
x_4=2\cdot x_1+x_2.
\]
Then the dependent triple $(0,4,-4)$ gives
\[
x_{-4}=-x_4=-2\cdot x_1-x_2.
\]
Now use the dependent triple $(-1,-4,5)$.  Again $\lambda=0$ gives
\[
x_{-1}+x_{-4}+x_5=0,
\]
hence
\[
-x_1+(-2\cdot x_1-x_2)+x_5=0,
\]
so
\[
x_5=3\cdot x_1+x_2.
\]
On the other hand, the dependent triple $(-2,-3,5)$ gives
\[
x_{-2}+x_{-3}+x_5=0.
\]
Substituting the known values,
\[
(-x_2)+(-x_1-x_2)+x_5=0,
\]
so
\[
x_5=x_1+2x_2.
\]
Comparing the two formulas for $x_5$, we obtain
\[
3\cdot x_1+x_2=x_1+2x_2,
\]
hence $x_2=2\cdot x_1$.  Substituting back, we conclude that
\[
x_3=3\cdot x_1,
\qquad
x_4=4\cdot x_1,
\qquad
x_5=5\cdot x_1,
\]
and therefore
\[
x_{-2}=-2\cdot x_1,
\qquad
x_{-3}=-3\cdot x_1,
\qquad
x_{-4}=-4\cdot x_1.
\]
Finally, the dependent triple $(0,5,-5)$ gives
\[
x_5+x_{-5}=0,
\]
so
\[
x_{-5}=-x_5=-5\cdot x_1.
\]
Together with $x_{-1}=-x_1$ and $x_0=0$, this proves that
\[
x_i=i\cdot x_1
\qquad\text{for all } i\in \{-5,-4,-3,-2,-1,0,1,2,3,4,5\}.
\]
\end{proof}

\subsection{The odd case}

We first treat the case in which $n$ is odd.

\begin{proposition}\label{prop:group-law-odd}
Assume that $n=2m+1\ge 13$ is odd. Then
\[
x_i=i\cdot x_1
\qquad\text{for all } i\in \Z/n\Z.
\]
Moreover, $x_1$ has exact order $n$ in the group $G(k)$.
\end{proposition}

\begin{proof}
By Lemma~\ref{lem:window-computation}, we already know that
\[
x_i=i\cdot x_1
\qquad\text{for all } i\in \{-5,-4,\dots,5\}.
\]

We now prove by induction that
\[
x_j=j\cdot x_1
\qquad\text{for } 0\le j\le m.
\]

For $j=0,1,2,3,4,5$ this is already known. Suppose inductively that
$x_j=j\cdot x_1$ for some integer $j$ with $5\le j\le m-1$.
Consider the two dependent triples
\[
(1,j,-j-1)
\qquad\text{and}\qquad
(0,j+1,-j-1).
\]
These are triples of distinct elements of $\Z/n\Z$, since
\[
1\le j\le m-1
\qquad\text{and}\qquad
n=2m+1.
\]
Because $\lambda=0$, Lemma~\ref{lem:line-section-constant} gives
\[
x_1+x_j+x_{-j-1}=0
\qquad\text{and}\qquad
x_{j+1}+x_{-j-1}=0.
\]
Subtracting the second equation from the first, we obtain
\[
x_{j+1}=x_j+x_1.
\]
By the inductive hypothesis, this becomes
\[
x_{j+1}=j\cdot x_1+x_1=(j+1)\cdot x_1.
\]

Thus, by induction,
\[
x_j=j\cdot x_1
\qquad\text{for all } 0\le j\le m.
\]

For negative indices, the dependent triple $(0,j,-j)$
gives
\[
x_j+x_{-j}=0,
\]
so
\[
x_{-j}=-x_j=-j\cdot x_1
\qquad\text{for } 1\le j\le m.
\]
Therefore
\[
x_i=i\cdot x_1
\qquad\text{for every } i\in \Z/n\Z.
\]

It remains to show that $x_1$ has exact order $n$. Consider the dependent triple $(m,m-1,2)$.
Since
\[
m+(m-1)+2=2m+1=n\equiv 0 \pmod n,
\]
and the three indices are distinct, Lemma~\ref{lem:line-section-constant} with $\lambda=0$ gives
\[
x_m+x_{m-1}+x_2=0.
\]
Substituting the already established formulas, we obtain
\[
m\cdot x_1+(m-1)\cdot x_1+2\cdot x_1 = n\cdot x_1 = 0.
\]
So $x_1$ is $n$-torsion.

Finally, the points $P_0,P_1,\dots,P_{n-1}$
are pairwise distinct in the realization, 
and by definition, $x_i\in C^\sm(k)$ is the group
element represented by the point $P_i$.
Hence the group elements
\[
0,x_1,2\cdot x_1,\dots,(n-1)\cdot x_1
\]
are pairwise distinct. Therefore the order of $x_1$ is exactly $n$.
\end{proof}

\subsection{The even case}

We next treat the even case.

\begin{proposition}\label{prop:group-law-even}
Assume that $n=2m\ge 12$ is even. Then
\[
x_i=i\cdot x_1
\qquad\text{for all } i\in \Z/n\Z.
\]
Moreover, $x_1$ has exact order $n$ in the group $G(k)$.
\end{proposition}

\begin{proof}
By Lemma~\ref{lem:window-computation}, we already know that
\[
x_i=i\cdot x_1
\qquad\text{for all } i\in \{-5,-4,\dots,5\}.
\]

We first prove by induction that
\[
x_j=j\cdot x_1
\qquad\text{for } 0\le j\le m-1.
\]

As before, the statement is already known for $j=0,1,2,3,4,5$. Suppose inductively that
$x_j=j\cdot x_1$
for some integer $j$ with
$5\le j\le m-2$.
Again, we consider the dependent triples
\[
(1,j,-j-1)
\qquad\text{and}\qquad
(0,j+1,-j-1).
\]
Since $j+1\le m-1$, these triples consist of distinct indices in $\Z/n\Z$. Since $\lambda=0$, we have
\[
x_1+x_j+x_{-j-1}=0
\qquad\text{and}\qquad
x_{j+1}+x_{-j-1}=0.
\]
Subtracting gives
\[
x_{j+1}=x_j+x_1=(j+1)\cdot x_1.
\]

Thus, by induction,
\[
x_j=j\cdot x_1
\qquad\text{for all } 0\le j\le m-1.
\]

For $1\le j\le m-1$, the dependent triple
$(0,j,-j)$
gives
\[
x_j+x_{-j}=0,
\]
hence
\[
x_{-j}=-j\cdot x_1.
\]

So only the half-period point $x_m$ remains. To determine it, use the dependent triples
\[
(m,1,m-1)
\qquad\text{and}\qquad
(m,-1,1-m).
\]
Since $\lambda=0$, these yield
\[
x_m+x_1+x_{m-1}=0
\qquad\text{and}\qquad
x_m+x_{-1}+x_{1-m}=0.
\]
Substituting the known formulas,
\[
x_m+x_1+(m-1)x_1=0,
\qquad
x_m-x_1-(m-1)x_1=0.
\]
Thus
\[
x_m=-m\cdot x_1
\qquad\text{and}\qquad
x_m=m\cdot x_1.
\]
Comparing these equalities gives
$(2m)\cdot x_1=0$, that is,
$n\cdot x_1=0$.
Therefore
\[
x_m=m\cdot x_1=-m\cdot x_1,
\] 
and hence
\[
x_i=i\cdot x_1
\qquad\text{for every } i\in \Z/n\Z.
\]

Finally, as in the odd case, the marked points $P_i$ are pairwise distinct, from which it follows that
the group elements
\[
0,x_1,2\cdot x_1,\dots,(n-1)\cdot x_1
\]
are pairwise distinct, and thus the order of $x_1$ is exactly $n$.
\end{proof}

Combining the odd and even cases, we obtain:

\begin{theorem}\label{thm:realization-standard}
Let $n\ge 12$, let $k$ be a field, and let
\[
(P_i)_{i\in \Z/n\Z}\subset \PP^2(k)
\]
be a realization of $\cT_n$. Then there exists an irreducible cubic $C\subset \PP^2_k$
such that:
\begin{enumerate}[label=\textup{(\roman*)}]
\item every point $P_i$ belongs to $C^\sm(k)$;
\item with the group law on $C^\sm$ having identity $P_0$, one has
\[
P_i=iP_1
\qquad\text{for all } i\in \Z/n\Z;
\]
\item the point $P_1$ has exact order $n$ in $C^\sm(k)$;
\item the embedding satisfies
\[
\mathcal O_C(1)\cong \mathcal O_C(3P_0).
\]
\end{enumerate}
\end{theorem}

\begin{proof}
The existence of an irreducible cubic containing all marked points and smooth at each marked point is Theorem~\ref{thm:cubic-irreducible-smooth}. Lemma~\ref{lem:window-computation} gives the line-bundle statement $\mathcal O_C(1)\cong \mathcal O_C(3P_0)$.
Proposition~\ref{prop:group-law-odd} proves the remaining assertions when $n$ is odd, and Proposition~\ref{prop:group-law-even} proves them when $n$ is even.
\end{proof}

\section{Completion of the field-valued correspondence}\label{sec:field-valued-bijection-proof}

We now assemble the preceding results to prove the field-valued classification theorem.

\begin{proof}[Proof of Theorem~\ref{thm:field-valued-bijection}]
Injectivity is Proposition~\ref{prop:field-valued-injection}, so it remains to prove surjectivity.
Let
\[
        (P_i)_{i\in \Z/n\Z}\subset \PP^2(k)
\]
be a realization of $\cT_n$.

First assume $n\ge 12$.  By Theorem~\ref{thm:realization-standard}, there is an
irreducible plane cubic $C\subset \PP^2_k$, smooth at all marked points, such that,
with respect to the group law on $C^\sm$ having identity $P_0$,
\[
        P_i=iP_1\quad\text{for all } i\in \Z/n\Z,
\]
the point $P_1$ has exact order $n$, and
\[
        \mathcal O_C(1)\cong \mathcal O_C(3P_0).
\]
It remains only to check that $C$ is either smooth or nodal.  Since a singular irreducible plane cubic is either nodal or cuspidal, it suffices to show that the cuspidal case is impossible. 
After base change to an algebraic closure, the smooth locus of a cuspidal cubic is isomorphic to the additive group $\G_a$.  
Since $\operatorname{char}(k)\nmid n$, multiplication by $n$ on $\G_a$ is an automorphism, contradicting 
the fact that $P_1$ has exact order $n$.  Therefore
$C$ is either smooth or nodal, and
\[
        (C,P_0,P_1)\in X_1(n)^\circ(k).
\]
By construction, its image under $\beta_k$ is the given realization.

The cases $n=10$ and $n=11$ are handled in
Subsections~\ref{subsec:n10} and~\ref{subsec:n11}, respectively.
\end{proof}

\section{Small levels}\label{sec:small-levels}

\subsection{The case \texorpdfstring{$n=10$}{n=10}}\label{subsec:n10}

\begin{lemma}\label{lem:n10-seven-subset}
Every seven-element subset of $\Z/10\Z$ contains three distinct elements whose sum
is zero.
\end{lemma}

\begin{proof}
Let $T\subset\Z/10\Z$ have cardinality $7$. If $0\in T$, then either $T$ contains
one of
\[
\{1,9\},\ \{2,8\},\ \{3,7\},\ \{4,6\},
\]
or it has at most $1+4+1=6$ elements, counting $0$, at most one element from each
pair, and possibly $5$. In the former case, adjoining $0$ gives the required triple.

If $0\notin T$ but $5\in T$, the same argument with
\[
\{1,4\},\ \{2,3\},\ \{6,9\},\ \{7,8\}
\]
shows that either $T$ contains a zero-sum triple through $5$, or $|T|\le 1+4=5$.

Finally, if $0,5\notin T$, then
\[
T=\{1,2,3,4,6,7,8,9\}\setminus\{e\}
\]
for some $e$. At least one of
\[
\{1,2,7\},\qquad \{1,3,6\},\qquad
\{3,8,9\},\qquad \{4,7,9\}
\]
omits $e$, and each displayed triple has sum zero modulo $10$.
\end{proof}

\begin{theorem}\label{thm:n10-bijective}
Let $k$ be a field with $\operatorname{char}(k)\nmid 10$. Then
\[
\beta_k\colon X_1(10)^\circ(k)\longrightarrow R_{\cT_{10}}(k)
\]
is a bijection.
\end{theorem}

\begin{proof}
Injectivity is Proposition~\ref{prop:field-valued-injection}, so it remains to prove
surjectivity. Let
\[
(P_i)_{i\in\Z/10\Z}\subset\PP^2(k)
\]
be a realization of $\cT_{10}$. By Corollary~\ref{cor:unique-window-cubic}, there is a
unique cubic $C\subset\PP^2_k$ through
\[
P_{-4},P_{-3},\ldots,P_5.
\]
These are all ten marked points.

We first prove that $C$ is irreducible. It cannot be supported on three lines, since
Lemma~\ref{lem:no-four-collinear} allows at most three marked points on each line.
If $C=Q\cup L$ with $Q$ an irreducible conic and $L$ a line, then $Q$ contains at least seven marked points.  By Lemma~\ref{lem:n10-seven-subset}, three of their indices are distinct and sum to zero, so the corresponding three points are collinear.  This is impossible on the irreducible conic $Q$ by B\'ezout's theorem.  Thus $C$ is irreducible. The proof of
Proposition~\ref{prop:smooth-points} now applies verbatim and shows that every $P_i$
is smooth on $C$.

As in the body of the paper, the existence of a smooth $k$-point implies that $C$ is geometrically
irreducible, so its smooth locus carries the group law with identity $P_0$. Write
$x_i\in C^\sm(k)$ for the element corresponding to $P_i$ and let
$H$ be the divisor cut out by a line. The calculation in the proof of
Lemma~\ref{lem:window-computation}, up to but excluding its final use of
$(0,5,-5)$, involves only the ten distinct indices $-4,-3,\ldots,5$, and every
triple used in that part of the calculation remains a triple of distinct residue classes
modulo $10$. It therefore applies and gives
\[
[H-3P_0]=0,
\qquad
x_i=i\cdot x_1 \quad(-4\le i\le 5).
\]
These indices represent every class modulo $10$. The dependent triple $(5,4,1)$
then gives
\[
10\cdot x_1=x_5+x_4+x_1=0.
\]
Since the ten marked points are distinct, the elements
$0,x_1,2x_1,\ldots,9x_1$ are distinct; hence $x_1$ has exact order $10$. Moreover,
$[H-3P_0]=0$ says that
\[
\mathcal O_C(1)\cong\mathcal O_C(3P_0).
\]

It remains only to exclude a cusp. The cuspidal-exclusion argument in
Section~\ref{sec:field-valued-bijection-proof} applies unchanged: after base change to an
algebraic closure, the smooth locus of a cuspidal cubic is $\G_a$, which has no
nonzero $10$-torsion because $\operatorname{char}(k)\nmid10$. Thus $C$ is smooth or
nodal, and $(C,P_0,P_1)$ is a preimage of the given realization under $\beta_k$.
\end{proof}

\subsection{The case \texorpdfstring{$n=11$}{n=11}}\label{subsec:n11}

\begin{lemma}\label{lem:n11-eight-subset}
Every eight-element subset of $\Z/11\Z$ contains three distinct elements whose sum
is zero.
\end{lemma}

\begin{proof}
There are
\[
\frac{11^2-3\cdot11+2}{6}=15
\]
unordered zero-sum triples of distinct elements of $\Z/11\Z$. The element $0$
occurs in five of them, and every nonzero element occurs in four (by scaling, since
$\Z/11\Z$ is a field). Let $S$ be the
three-element complement of an eight-element set $T$. If $0\notin S$, then $S$ meets
at most $3\cdot4=12$ of the fifteen triples. If $0\in S$, then it meets five, and each
of the other two elements of $S$ meets at most three additional triples, for a total
of at most $11$. In either case, some zero-sum triple is disjoint from $S$ and hence
contained in $T$.
\end{proof}

\begin{theorem}\label{thm:n11-bijective}
Let $k$ be a field with $\operatorname{char}(k)\nmid 11$. Then
\[
\beta_k\colon X_1(11)^\circ(k)\longrightarrow R_{\cT_{11}}(k)
\]
is a bijection.
\end{theorem}

\begin{proof}
Again, injectivity is Proposition~\ref{prop:field-valued-injection}. Let
\[
(P_i)_{i\in\Z/11\Z}\subset\PP^2(k)
\]
be a realization of $\cT_{11}$. Corollary~\ref{cor:unique-window-cubic} gives a
unique cubic $C$ through $P_{-4},\ldots,P_5$. The grid
\[
\Gamma(-5,0,1,3)=
\begin{pmatrix}
-5&0&5\\
1&3&-4\\
4&-3&-1
\end{pmatrix}
\]
has nine distinct entries modulo $11$, eight of which lie in this window. By
Lemma~\ref{lem:grid-completion}, $P_{-5}\in C$. Hence $C$ contains all eleven marked
points.

As before, $C$ cannot be supported on three lines. If $C=Q\cup L$ with $Q$ an
irreducible conic, then $Q$ contains at least eight marked points.  By Lemma~\ref{lem:n11-eight-subset}, three of their indices are distinct and sum to zero, so the corresponding collinear triple lies on $Q$, contradicting B\'ezout's theorem.  Hence $C$ is irreducible, and
Proposition~\ref{prop:smooth-points} shows that all marked points lie in $C^\sm(k)$.

Equip $C^\sm$ with identity $P_0$ and write $x_i$ for the element corresponding to
$P_i$. Since the eleven indices $-5,-4,\ldots,5$ are pairwise
distinct modulo $11$, Lemma~\ref{lem:window-computation} gives
\[
\mathcal O_C(1)\cong\mathcal O_C(3P_0),
\qquad
x_i=i\cdot x_1 \quad(-5\le i\le5).
\]
These are all residue classes modulo $11$. The dependent triple $(5,4,2)$ gives
$11x_1=0$, and the distinctness of the eleven marked points shows that $x_1$ has exact
order $11$.

Finally, the cuspidal case is excluded exactly as in
Section~\ref{sec:field-valued-bijection-proof}, since multiplication by $11$ is invertible
on $\G_a$. Thus $C$ is smooth or nodal, and $(C,P_0,P_1)$ maps to the given
realization.
\end{proof}

\subsection{The range \texorpdfstring{$n\le9$}{n<=9}}
\label{sec:small-n}

The situation for smaller $n$ was
analyzed over $\C$ by Borisov and Roulleau \cite[\S3]{BR}, whose treatment we
briefly summarize; we refer the reader to their paper for details and for explicit
equations.

For $n\le3$, the matroid $\cT_n$ is not well-defined, so we assume $n \geq 4$.

For $4 \leq n\le6$, the $k$-realization space of $\cT_n$ is a single point for every field $k$,
so the map $\beta$ is constant, whereas $X_1(n)^\circ$ is a curve.
This can be verified by explicit computation.
For example, when $n=5$, the four elements $1,2,3,4$ are in general position, and once the
corresponding points are placed in the standard frame
$(1:0:0),(0:1:0),(0:0:1),(1:1:1)$ the non-bases $\{0,1,4\}$ and $\{0,2,3\}$
force $P_0=(0:1:1)$. 

For $7\le n\le9$, the realization space of $\cT_n$ is a rational curve, and
\cite[Theorem~1]{BR} asserts that $\beta$ is birational for every 
algebraically closed field $k$ of characteristic not dividing $n$.  In fact,
$\beta$ is an isomorphism in this range: a direct computation shows
that $\cR_{\cT_n}$ is the complement of $3$ (resp. $4$, $5$) points in
$\PP^1$ for $n=7$ (resp. $n=8,9$), which is exactly the number of cusps of $X_1(n)$
carrying a reducible N\'eron polygon, so that $X_1(n)^\circ$ and
$\cR_{\cT_n}$ have the same number of punctures. It follows easily that 
$\beta$ is in fact an isomorphism for every field $k$ of characteristic not dividing $n$.
What fails for $7\le n\le9$ is therefore not the conclusion but our method.

The reason for the hypothesis $n \geq 10$ in our argument is the following.  
The reconstruction of \Cref{sec:seed} requires ten marked points: nine of them form the base
locus of a pencil of cubics, and the tenth selects a particular member of that pencil
(\Cref{cor:unique-window-cubic}).  For $n\le9$, there is no tenth point.  The
case $n=9$ is the extreme one: the whole ground set $\Z/9\Z$ is a single
zero-sum grid in the sense of \Cref{def:gamma-grid}, namely
\[
        \Gamma(0,1,2,3)=
        \begin{pmatrix} 0&1&8\\ 2&3&4\\ 7&5&6\end{pmatrix},
\]
so the nine marked points of any realization of $\cT_9$ are precisely the
complete intersection of the triangle of row-lines with the triangle of
column-lines.  By Chasles' theorem
(\Cref{thm:CBC}), they impose only eight conditions on cubics, and the
cubics through them form a pencil with no distinguished member.  A
realization of $\cT_n$ with $n\le9$ therefore need not determine a particular plane
cubic, and there are no extra points to propagate.

Borisov and Roulleau also observe that a birationality assertion for $n=5,6$ can be recovered by
enlarging the matroid: they construct auxiliary matroids,
obtained from the tangent lines and secant lines of the configuration, whose
realization spaces are birational to $X_1(5)^\circ$ and $X_1(6)^\circ$
\cite[Remarks 3.1 and 3.2]{BR}.  They further suggest that embedding $E$ by
$|kO|$ for $k\ge4$ should realize $X_1(n)$ as the realization space of a
matroid of rank $k$; we have not pursued this direction.

\subsection{An abstract form of the group-law calculation}

The calculation used above will also be needed later over more general local Artinian rings.  
We therefore record its purely group-theoretic content.

\begin{lemma}[Abstract group-law reconstruction]\label{lem:abstract-group-law}
Let $n\ge 10$, let $G$ be an abelian group, and let $x_i\in G$ be indexed by
$i\in\Z/n\Z$.  Suppose that $x_0=0$ and that there is an element $\lambda\in G$
such that
\[
        x_a+x_b+x_c=\lambda
\]
for every triple of distinct indices $a,b,c\in\Z/n\Z$ satisfying
$a+b+c=0$.  Then
\[
        \lambda=0,\qquad x_i=i\cdot x_1\quad(i\in\Z/n\Z),
        \qquad n\cdot x_1=0.
\]
\end{lemma}

\begin{proof}
For $n\ge12$, the proofs of Lemma~\ref{lem:window-computation} and
Propositions~\ref{prop:group-law-odd} and~\ref{prop:group-law-even} use only
the displayed relations in the abelian group $G$; the geometry enters solely
to produce those relations.  They therefore prove the assertion verbatim.
For $n=10$, the truncated window calculation in the proof of
Theorem~\ref{thm:n10-bijective}, followed by the relation associated to
$(5,4,1)$, gives the result.  For $n=11$, Lemma~\ref{lem:window-computation}
and the relation associated to $(5,4,2)$ give the result.
\end{proof}

\part{The scheme-theoretic correspondence}\label{part:scheme}

\section{The canonical realization scheme of a matroid}
\label{sec:schemes}

In this section, we construct the realization scheme $\cR_M$ of a matroid $M$ as an affine scheme over $\Z$.

Let $E$ be a finite totally ordered set, and let $M$ be a matroid of
rank $r\geq 2$ on $E$.  
Write $\mathcal B(M)$ for its set of bases.  For an
$(r-2)$-subset $I$ and distinct elements $i,j,k,l\notin I$, we use the usual
alternating convention for symbols $p_{Iij}$, and set $p_J=0$ when $J$ is not
a basis of $M$.

Let
\[
P_{r,E}:=\Z[p_I:I\in\tbinom Er]
\]
and let $I_{r,E}\subset P_{r,E}$ be the Pl\"ucker ideal.  
Define a homomorphism
\[
\varphi_M:P_{r,E}\longrightarrow
C_M=\Z[p_B^{\pm1}:B\in\mathcal B(M)]
\]
by
\[
\varphi_M(p_I)=
\begin{cases}
p_I,& I\in\mathcal B(M),\\
0,& I\notin\mathcal B(M).
\end{cases}
\]
Let
\[
J_M:=\varphi_M(I_{r,E})\,C_M.
\]

\begin{definition}
The \emph{Pl\"ucker ring} of $M$ is the ring 
\[
S_M:=C_M/J_M.
\]
\end{definition}

Let $G_M:=\Z\oplus\Z^E$ and grade the Pl\"ucker ring $S_M$ by
\[
        \mdeg(p_B)=\left(1,\sum_{e\in B}e_e\right) \in G_M.
\]

\begin{definition}
\label{def:degree-zero-ring}
Define the \emph{multidegree-zero Pl\"ucker ring} of $M$ by $A_M:=(S_M)_0$,
and define the \emph{realization scheme} of $M$ to be
\[
\cR_M := \Spec A_M.
\]
\end{definition}

Changing the auxiliary total order on $E$ multiplies the alternating
Pl\"ucker coordinates by the corresponding signs and induces a
canonical multigraded isomorphism of the resulting rings.  Thus
$A_M$ and $\cR_M$ are independent of this choice, up to canonical
isomorphism.  

For the construction of $\cR_M$ to be well-behaved, the following observation is important.

\begin{proposition}\label{prop:finite-generation-AM}
The ring $A_M$ is a finitely generated $\Z$-algebra.
\end{proposition}

\begin{proof}
The ring
\[
C_M=\Z[p_B^{\pm1}:B\in\mathcal B(M)]
\]
is the group algebra of the free abelian group
$\Z^{\mathcal B(M)}$.

Let $\varepsilon_B$ denote the standard basis vector corresponding to
$B\in\mathcal B(M)$.  The multigrading on $C_M$ is induced by the
group homomorphism
\[
\delta_M:\Z^{\mathcal B(M)}\longrightarrow G_M=\Z\oplus\Z^E
\]
defined by
\[
\delta_M(\varepsilon_B)
=
\left(1,\sum_{e\in B}e_e\right).
\]

The Laurent monomial
\[
p^a:=\prod_{B\in\mathcal B(M)}p_B^{a_B}
\]
has multidegree $\delta_M(a)$.

Let $L:=\ker(\delta_M).$
Since $L$ is a subgroup of the finitely generated free abelian group
$\Z^{\mathcal B(M)}$, it is itself a finitely generated free abelian
group.  The multidegree-zero part of $C_M$ is the group algebra of
$L$: $(C_M)_0=\Z[L]$.
After choosing a basis of $L$, this is a finitely generated Laurent
polynomial algebra over $\Z$.
Since $J_M$ is multihomogeneous, the quotient map
\[
C_M\longrightarrow S_M=C_M/J_M
\]
induces a surjection
\[
(C_M)_0\longrightarrow(S_M)_0=A_M.
\]
Thus $A_M$ is a quotient of a finitely generated $\Z$-algebra, and is
therefore itself finitely generated over $\Z$.
\end{proof}

\section{The elliptic realization scheme}
\label{sec:canonical-Tn}

We now specialize the preceding construction to the elliptic matroid $\cT_n$.

\begin{definition}\label{def:elliptic-realization-scheme}
For $n\ge10$, define the \emph{realization scheme} of $\cT_n$ to be
\[
        \cR_n:=\cR_{\cT_n}\times_{\Spec\Z}\Spec\Z[1/n]
        =\Spec A_{\cT_n}[1/n].
\]
\end{definition}

\begin{definition}\label{def:unimodular}
Let $A$ be a local ring.  A vector $q=(a,b,c)\in A^3$
is called \emph{unimodular} if its coordinates generate the unit
ideal: $(a,b,c)=A$.

Since $A$ is local, this is equivalent to requiring that at least one
of $a,b,c$ be a unit.
Every section $Q\in\PP^2_A(A)$ can be represented by a unimodular
homogeneous coordinate vector $q\in A^3$, unique up to multiplication
by an element of $A^\times$.
\end{definition}

\begin{definition}\label{def:projective-frame}
An ordered quadruple $(Q_0,Q_1,Q_2,Q_3)$
of sections of $\PP^2_A\to\Spec A$ is a \emph{projective frame} if,
after choosing unimodular homogeneous coordinate vectors
$q_i\in A^3$ representing the $Q_i$, every determinant
\[
\det(q_i,q_j,q_k),
\qquad 0\le i<j<k\le3
\]
is a unit in $A$.
This condition is independent of the chosen lifts and is
equivalent to requiring that no three of the four points in the closed fiber are collinear.  
\end{definition}

The \emph{standard normalized projective frame} is
\[
        (1:0:0),\quad(0:1:0),\quad(0:0:1),\quad(1:1:1).
\]

When a labelled configuration $(P_i)_{i\in E}$ is under discussion, we call
a four-element subset $F=\{i_0,i_1,i_2,i_3\}\subset E$, with a specified
ordering of its elements, a \emph{frame-label set} if the ordered quadruple
$(P_{i_0},P_{i_1},P_{i_2},P_{i_3})$ is a projective frame.  

\subsection{The normalized frame slice}
\label{subsec:normalized-frame-slice}

We now compare the multidegree-zero Pl\"ucker ring $A_{\cT_n}$ with the normalized representation scheme used in the deformation arguments below.

Let
\[
        E=\Z/n\Z,
        \qquad
        F=\{0,1,2,3\}.
\]
We use $F$ as our frame-label set.  Since every three-element subset of $F$
is a basis of $\cT_n$, the ordered quadruple
$(P_0,P_1,P_2,P_3)$ is a projective frame in every realization of
$\cT_n$.  We normalize it to the standard projective frame
\[
        P_0=(1:0:0),\quad
        P_1=(0:1:0),\quad
        P_2=(0:0:1),\quad
        P_3=(1:1:1).
\]
For an ordered triple $(a,b,c)$ of distinct elements of $E$, we write $p_{abc}$ for the corresponding alternating Pl\"ucker coordinate: if $I=\{a,b,c\}$ and $i<j<k$ are the elements of $I$ in increasing order, then
\[
        p_{abc}=\operatorname{sgn}(a,b,c;i,j,k)p_{ijk}.
\]

For $i\in E\setminus F$, at least one of the triples
\[
        \{i,1,2\},\qquad \{0,i,2\},\qquad \{0,1,i\}
\]
is a basis of $\cT_n$, since these triples have sums $i+3$, $i+2$, and $i+1$, respectively.  Choose once and for all one such basis and denote it by $B_i$.  Since $B_i$ is a basis of $\cT_n$, the Pl\"ucker coordinate $p_{B_i}$ is inverted in the Pl\"ucker ring $S_{\cT_n}$ of $\cT_n$.

Let $A_{\cT_n}=(S_{\cT_n})_0$
be the multidegree-zero Pl\"ucker ring of $\cT_n$, and let $W$ be the following set of multihomogeneous units of $S_{\cT_n}$:
\[
        p_{012},\quad p_{013},\quad p_{023},\quad p_{123},
        \quad\text{and}\quad p_{B_i}\;(i\in E\setminus F).
\]

\begin{lemma}
The multidegrees of the elements of $W$ form a basis for the subgroup of $\Z\oplus\Z^E$ generated by the multidegrees of all Pl\"ucker coordinates of $\cT_n$.  
\end{lemma}

\begin{proof}
Write
\[
        d_{abc}:=\mdeg(p_{abc})=(1,e_a+e_b+e_c).
\]
The four multidegrees associated with the frame-label set, namely $d_{012},d_{013},d_{023},d_{123}$, generate one Pl\"ucker degree and all differences $e_a-e_b$ with $a,b\in F$.  If $B_i=\{i,a,b\}$ with $a,b\in F$, choose any $c\in F\setminus\{a,b\}$. Then
\[
        d_{iab}-d_{cab}=e_i-e_c.
\]
Thus the multidegrees in $W$ generate $e_i-e_j$ for all labels $i,j\in E$, together with one Pl\"ucker degree.  Since every Pl\"ucker degree is obtained from $d_{012}$ by adding such label differences, these multidegrees generate the subgroup in question.  They are also independent: in a relation among the multidegrees in $W$, the coefficient of $e_i$ for $i\in E\setminus F$ is exactly the coefficient of $d_{B_i}$, so all these coefficients vanish; the remaining relation among $d_{012},d_{013},d_{023},d_{123}$ is forced to be trivial by comparing the four coordinates indexed by the frame-label set.  Hence the multidegrees in $W$ form a basis.
\end{proof}

Let $N_n$ be the affine scheme over $\Z$ obtained as follows.
 For each $i\in E\setminus F$, write $P_i=(X_i:Y_i:Z_i)$ and impose the affine chart condition corresponding to $B_i$:
\[
\begin{array}{ll}
X_i=1, & \text{if } B_i=\{i,1,2\},\\
Y_i=1, & \text{if } B_i=\{0,i,2\},\\
Z_i=1, & \text{if } B_i=\{0,1,i\}.
\end{array}
\]
Together with the standard normalized projective frame $(P_0,P_1,P_2,P_3)$, these coordinates give a $3\times n$ matrix.  The coordinate ring of $N_n$ is the polynomial ring in the remaining affine coordinates modulo the equations saying that all nonbasis $3\times3$ determinants vanish, localized at all basis determinants.  Thus $N_n$ represents normalized labelled configurations with support $\cT_n$ in this fixed normalized projective-frame chart.

\begin{remark}
The way we have defined $N_n$ is the approach taken in \cite{BR} to \emph{defining} the realization scheme of $\cT_n$. We have chosen to give a more canonical construction of 
the realization scheme and to interpret $N_n$ as a convenient presentation, rather than as a definition.
\end{remark}

\begin{proposition}
\label{prop:normalized-frame-slice}
For the fixed choice of the basis triples $B_i$ above, there is a natural isomorphism
\[
        A_{\cT_n}\;\xrightarrow{\sim}\;\Gamma(N_n,\mathcal O_{N_n}).
\]
\end{proposition}

Before giving the proof, we need the following ring-theoretic lemma.

\begin{lemma} 
\label{lem:graded-ring}
Let $S$ be a ring graded by a free abelian group $\Lambda$, and suppose that $u_1,\ldots,u_m$ are homogeneous units whose degrees form a basis for the subgroup of $\Lambda$ generated by the degrees occurring in $S$.  Then multiplication by monomials in the $u_j$ gives an isomorphism
\[
        S_0[u_1^{\pm1},\ldots,u_m^{\pm1}]\;\xrightarrow{\sim}\; S.
\]
\end{lemma}

\begin{proof}
If $f\in S$ is homogeneous of degree $\lambda$, there is a unique integer vector $a=(a_1,\ldots,a_m)$ such that $\lambda=\sum a_j\deg(u_j)$; then $f u^{-a}\in S_0$.  This gives the inverse map on homogeneous elements and hence on $S$.  In particular, the quotient of $S$ by the relations $u_j=c_j$, where each $c_j$ is a prescribed unit, is canonically identified with $S_0$ after rescaling by the constants $c_j$.
\end{proof}

\begin{proof}[Proof of Proposition~\ref{prop:normalized-frame-slice}]
Apply Lemma~\ref{lem:graded-ring} to $S=S_{\cT_n}$ and to the set $W$ above.  It follows that the ring $A_{\cT_n}$ is obtained from $S_{\cT_n}$ by pinning the units in $W$ to their normalized values.  These normalizations are exactly the equations fixing the ordered projective frame to the standard normalized frame
\[
        P_0=(1:0:0),\quad P_1=(0:1:0),\quad P_2=(0:0:1),\quad P_3=(1:1:1)
\]
and the affine chart equations $X_i=1$, $Y_i=1$, or $Z_i=1$ chosen above for each $i\in E\setminus F$.

We now identify the resulting quotient with the coordinate ring of $N_n$.  On the normalized chart, every Pl\"ucker coordinate $p_I$ is sent to the corresponding $3\times3$ determinant $\Delta_I$ of the normalized matrix.  The ordinary Pl\"ucker relations hold because the $\Delta_I$ are actual minors of a matrix.  The equations $p_I=0$ for nonbases become exactly the determinant-vanishing equations defining the support $\cT_n$, and the inversion of the basis Pl\"ucker coordinates becomes exactly the localization at all basis determinants.

Conversely, the affine coordinates of the normalized matrix can be recovered from Pl\"ucker coordinates.  With the alternating convention above, the three homogeneous coordinates of $P_i$ relative to the first three points of the standard projective frame are:
\[
        X_i= \pm\frac{p_{i12}}{p_{012}},\qquad
        Y_i= \pm\frac{p_{0i2}}{p_{012}},\qquad
        Z_i= \pm\frac{p_{01i}}{p_{012}},
\]
where the signs depend only on the chosen ordering convention.  After normalizing the chosen coordinate to $1$, the remaining two affine coordinates are the corresponding ratios of the three displayed Pl\"ucker coordinates.  A displayed ratio need not itself have multidegree zero; however, the basis property of the multidegrees of $W$ supplies a unique Laurent monomial in the pinned units which corrects it to multidegree zero.  After the units in $W$ are set equal to their normalized values, the corrected expression agrees with the displayed ratio.  Thus the determinant map and the coordinate-recovery map are inverse to one another, with the signs absorbed by the fixed units $\pm1$.
Therefore $A_{\cT_n}\cong\Gamma(N_n,\mathcal O_{N_n})$.  
\end{proof}

Finally, we note for later reference:

\begin{proposition}
\label{prop:Rn-field-points}
For every field $k$ with $\operatorname{char}(k)\nmid n$, the set $\cR_n(k)$ is naturally identified with the set of rescaling classes of $k$-representations of $\cT_n$.
\end{proposition}

\begin{proof}
The normalized frame slice of \cref{prop:normalized-frame-slice} identifies $\cR_n$ with the normalized representation scheme.  Every rescaling class of a $k$-representation of $\cT_n$ has a unique representative in which the ordered projective frame $(P_0,P_1,P_2,P_3)$ attached to the frame-label set $F=\{0,1,2,3\}$ is normalized; column rescalings then choose the affine representatives used in the construction of $N_n$.  Under this representative, the functions in $A_{\cT_n}$ are exactly the multidegree-zero Pl\"ucker coordinates.  This gives the asserted identification.
\end{proof}

\section{The modular curve and the modular morphism}
\label{sec:modular-morphism}

\subsection{The open modular curve \texorpdfstring{$X_1(n)^\circ$}{X1(n)-circ} and its cusps}

We use the compactified modular curve of generalized elliptic curves with
$\Gamma_1(n)$-level structure and take as known facts the existence of $X_1(n)$, its universal generalized
elliptic curve, and the representability of the $\Gamma_1(n)$-moduli problem;
see Deligne--Rapoport \cite{DR}, Katz--Mazur \cite[Chs.~3--5, 8]{KM}, and
Conrad \cite{Conrad-KM}.

For $n\ge5$, the moduli problem in question is
represented by a scheme over $\Z[1/n]$.  Let $X_1(n)$ denote the compactified
modular curve, and let $Y_1(n)\subset X_1(n)$ be the usual open modular curve
parametrizing smooth elliptic curves with a $\Gamma_1(n)$-level structure.  We define
\[
        X_1(n)^\circ\subset X_1(n)
\]
to be the open locus on which the geometric fibers of the universal generalized
elliptic curve are irreducible.  Thus the fibers over $X_1(n)^\circ$ are
smooth elliptic curves and N\'eron $1$-gons, and $Y_1(n)\subset X_1(n)^\circ$.  The relative complete linear
system $|3O|$ embeds every such fiber as a smooth or irreducible nodal plane
cubic.  Consequently, the geometric points of $X_1(n)^\circ$ are precisely
the triples $(C,O,P)$ used in Part~\ref{part:field}.  More generally, since
$n\ge5$ and the $\Gamma_1(n)$ moduli problem is representable, for every
field $k$ over $\Z[1/n]$ the set $X_1(n)^\circ(k)$ is naturally the set of
projective isomorphism classes used in Part~\ref{part:field}.

\begin{remark}
\label{rem:why-invert-n}
The restriction to $\Z[1/n]$ in the statement of
\Cref{thm:main} is not imposed merely to simplify the deformation
argument.  At a prime $p\mid n$, a $\Gamma_1(n)$-structure is a
Drinfeld level structure.  On the irreducible locus relevant here, a
section $P$ has exact order $n$ in the Drinfeld sense when the
effective Cartier divisor $\sum_{i\in\Z/n\Z}[iP]$
is a finite flat subgroup scheme of rank $n$ (the usual ampleness
condition is automatic for a N\'eron $1$-gon, see
\cite[\S2.4]{Conrad-KM}).  In particular, the geometric sections $iP$
need not be pairwise distinct.

For example, let $k$ be an algebraically closed field of
characteristic $p$, write
\[
        n=p^a m
        \qquad\text{with}\qquad
        (m,p)=1,
\]
and identify the smooth locus of a N\'eron $1$-gon with $\G_m$.
If $P=\zeta_m$ is a primitive $m$-th root of unity, then
\[
        \sum_{i=0}^{n-1}[P^i]
        =
        p^a\sum_{\xi^m=1}[\xi]
\]
is the subgroup scheme $\mu_n$, since
\[
        x^n-1=(x^m-1)^{p^a}
\]
in characteristic $p$.  Thus $P$ has exact order $n$ in the
Drinfeld sense, although its multiples give only $m$ distinct
geometric points.  They therefore do not form a realization of the
simple matroid $\cT_n$.

Consequently, over primes dividing $n$, the universal level sections
do not directly define a morphism to the realization
scheme $\cR_{\cT_n}$ by the Pl\"ucker-coordinate construction used
below.  An integral comparison would require a separate analysis of
the bad fibers, and possibly a modification of the modular model or
an enlargement of the realization space which allows the matroid
support to degenerate. We do not pursue this direction in the present paper.
\end{remark}

\subsection{Construction of the modular morphism}

Let
\[
        \pi:\cC\longrightarrow X_1(n)^\circ
\]
be the universal smooth-or-nodal irreducible generalized elliptic curve, with
identity section $O$ and exact order-$n$ section $P$.  For
$i\in\Z/n\Z$, put $P_i=iP\in\cC^{\sm}$.
        
The complete linear system $|3O|$ embeds $\cC$ in a relative projective plane.
On every geometric fiber, the labelled sections $(P_i)$ have Pl\"ucker
support $\cT_n$.  Since $X_1(n)^\circ$ is smooth, and hence reduced, over
$\Z[1/n]$, the nonbasis Pl\"ucker sections vanish identically because they
vanish on every geometric fiber; the basis Pl\"ucker sections are nowhere
vanishing.  Every multidegree-zero Pl\"ucker expression is independent of
the local trivializations of the rank-three linear system and of the homogeneous lifts
of the individual sections, and therefore defines a regular function on
$X_1(n)^\circ$.  By the universal property of
$\cR_n=\Spec A_{\cT_n}[1/n]$, these functions define a morphism
\[
        \beta:X_1(n)^\circ\longrightarrow\cR_n.
\]

\begin{proposition}\label{prop:beta-compatible-part1}
For every field $k$ with $\operatorname{char}(k)\nmid n$, the map
\[
        \beta(k):X_1(n)^\circ(k)\longrightarrow\cR_n(k)
\]
agrees, under Proposition~\ref{prop:Rn-field-points}, with the field-valued
map of Part~\ref{part:field}.
\end{proposition}

\begin{proof}
Both maps send $(C,O,P)$ to the multidegree-zero Pl\"ucker coordinates,
or equivalently to the rescaling class, of the labelled configuration $(iP)_{i\in\Z/n\Z}$.
\end{proof}

\begin{corollary}\label{cor:geometric-bijective}
For every field $k$ with $\operatorname{char}(k)\nmid n$, the map
\[
        \beta(k):X_1(n)^\circ(k)\longrightarrow\cR_n(k)
\]
is bijective.
\end{corollary}

\begin{proof}
Combine Theorem~\ref{thm:field-valued-bijection} with
Propositions~\ref{prop:Rn-field-points} and~\ref{prop:beta-compatible-part1}.
\end{proof}

\section{An Artinian-point criterion and normalized deformations}
\label{sec:normalization}

\subsection{The Artinian-point criterion}
\label{sec:artinian-point-criterion}

The following standard consequence of the infinitesimal criterion for
\'etaleness explains why we now focus on establishing bijectivity for points over local Artinian
rings. 
This result will allow us to pass directly from the Artinian reconstruction theorem 
(\Cref{thm:artinian-reconstruction}) to the global scheme-theoretic correspondence.  

\begin{proposition}[Artinian-point criterion for isomorphisms]
\label{prop:artinian-point-criterion}
Let $S=\Spec R$, and let $f:X\longrightarrow Y$
be a morphism locally of finite presentation between locally
Noetherian $S$-schemes.  For an $R$-algebra $A$, write
\[
X(A):=\Hom_S(\Spec A,X)
\]
and similarly for $Y(A)$.  Then the following are equivalent:
\begin{enumerate}[label=\textup{(\roman*)}]
\item
The morphism $f$ is an isomorphism.
\item
For every local Artinian $R$-algebra $A$, the induced map
\[
f(A):X(A)\longrightarrow Y(A)
\]
is bijective.
\end{enumerate}
\end{proposition}

\begin{proof}
The implication
\[
\textup{(i)}\Longrightarrow\textup{(ii)}
\]
is immediate.  Suppose that \textup{(ii)} holds.

We first show that $f$ is \'{e}tale.  Let $A'\longrightarrow A$
be one of the small extensions of local Artinian $R$-algebras
appearing in the local Artinian criterion for \'{e}taleness.  Suppose
that we are given $x\in X(A)$ and $y'\in Y(A')$ such that
\[
f(x)=y'|_A
\]
in $Y(A)$.  Since
\[
f(A'):X(A')\longrightarrow Y(A')
\]
is surjective, there is an element $x'\in X(A')$ such that
\[
f(x')=y'.
\]
The restriction $x'|_A$ and the given point $x$ have the same image
in $Y(A)$.  Since
\[
f(A):X(A)\longrightarrow Y(A)
\]
is injective, we have
\[
x'|_A=x.
\]
Moreover, the lift $x'$ is unique because $f(A')$ is injective.
Thus $f$ satisfies the unique local Artinian lifting criterion at
every point.  Since $f$ is locally of finite presentation and $Y$ is
locally Noetherian, the infinitesimal criterion for \'{e}taleness
shows that $f$ is \'{e}tale; see
\cite[Proposition~17.14.2]{EGAIV4}.

Taking $A=k$ for an arbitrary field $k$ equipped with an
$R$-algebra structure shows that $f(k):X(k)\longrightarrow Y(k)$
is bijective for every such field.  Its injectivity for every field
implies that $f$ is universally injective, or equivalently radicial;
see \cite[Tag~01S4]{Stacks}.  Its surjectivity for every field also
implies that $f$ is surjective: for any point $y\in Y$, apply the
surjectivity of $f(\kappa(y))$ to the canonical
$\kappa(y)$-valued point of $Y$ determined by $y$.

A universally injective \'{e}tale morphism is an open immersion
\cite[Tag~02LC]{Stacks}.  Since $f$ is also surjective, this open
immersion is an isomorphism.
\end{proof}

\begin{remark}
\label{rem:algebraically-closed-artinian-tests}
A standard geometric-point refinement of the preceding criterion
allows one to test only local Artinian $R$-algebras whose residue
fields are algebraically closed.  
\end{remark}

By Corollary~\ref{cor:geometric-bijective}, the map $\beta$ is already
bijective over every field of characteristic prime to $n$.  
Our goal is to bootstrap this field-theoretic bijectivity to the statement that
\[
        \beta(A):X_1(n)^\circ(A)\longrightarrow\cR_n(A)
\]
is bijective for every local Artinian $\Z[1/n]$-algebra $A$.  We will do this by
fixing the reduction of an $A$-point and comparing the corresponding
deformation problems through normalized representatives.

For the remainder of this section, $A$ denotes a local Artinian
$\Z[1/n]$-algebra and
\[
        k:=A/\mathfrak m_A
\]
denotes its residue field.  We do not assume that 
$A$ contains a coefficient field; this allows the deformation
argument to retain mixed-characteristic directions.

\subsection{Projective normalization over Artinian rings}

\begin{lemma}[Projective normalization]
\label{lem:projective-normalization-artinian}
Let $Q_0,Q_1,Q_2,Q_3$ be sections of $\PP^2_A$ whose ordered special-fiber quadruple is a projective frame in the sense of Definition~\ref{def:projective-frame}.  Then there is a unique element of $\PGL_3(A)$ carrying them to
\[
        (1:0:0),\quad (0:1:0),\quad (0:0:1),\quad (1:1:1).
\]
\end{lemma}

\begin{proof}
Choose lifts $v_i\in A^3$ for the sections $Q_i$.  Since the special
fibers of $Q_0,Q_1,Q_2$ are linearly independent, the matrix with columns
$v_0,v_1,v_2$ has unit determinant.  Its inverse carries these three sections
to the coordinate points.  This transformation is unique up to left
multiplication by a diagonal element of $\PGL_3(A)$.  The image of $Q_3$ has
homogeneous coordinates $(a:b:c)$ with $a,b,c\in A^\times$, because its
special fiber lies on none of the coordinate lines.  There is a unique
diagonal element modulo scalars that sends $(a:b:c)$ to $(1:1:1)$.  The
composite is therefore the unique projective transformation with the stated
property.

\end{proof}

\subsection{Normalized representation deformations}

Let $\bar x\in \cR_n(k)$ and let $\bar p=(P_i)$ be the corresponding normalized $k$-realization.  The point $\bar x$ lies on the normalized chart determined by the basis $\{0,1,2\}$ and the frame-label set $F=\{0,1,2,3\}$; equivalently, the ordered projective frame $(P_0,P_1,P_2,P_3)$ has been put in standard normalized form.  On this chart, the degree-zero Pl\"ucker coordinates are equivalent to the affine coordinates of a unique normalized labelled configuration.

\begin{definition}
A \emph{normalized $A$-valued deformation} of $\bar p$ is a labelled collection of sections
\[
        (\widetilde P_i)_{i\in\Z/n\Z}\subset \PP^2_A
\]
reducing to $\bar p$, satisfying the determinant vanishing and nonvanishing conditions of $\cT_n$, and such that
\[
        \widetilde P_0=(1:0:0),\quad
        \widetilde P_1=(0:1:0),\quad
        \widetilde P_2=(0:0:1),\quad
        \widetilde P_3=(1:1:1).
\]
\end{definition}

Whenever we pass from these projective sections to matrix coordinates, we use
for each $i\notin F$ the unique unimodular representative whose coordinate
corresponding to the chosen basis $B_i$ is equal to $1$.  This is possible
because that coordinate is a unit, and it is precisely the affine chart
normalization used in the definition of $N_n$.

\begin{proposition}
\label{prop:canonical-normalized-deformations}
Let $\bar x\in \cR_n(k)$ correspond to the normalized representation
$\bar p$.  For every local Artinian $\Z[1/n]$-algebra $A$ with residue
field $k$, deformations of $\bar x$ to $\cR_n(A)$ are naturally in bijection with 
normalized $A$-valued deformations of $\bar p$.
\end{proposition}

\begin{proof}
After base change to $\Z[1/n]$, Proposition~\ref{prop:normalized-frame-slice}
identifies $\cR_n$ with the normalized representation scheme $N_n$.
Therefore an $A$-valued deformation of $\bar x$ is exactly an $A$-point of
$N_n$ reducing to $\bar p$: a labelled configuration satisfying the
determinant vanishing and nonvanishing conditions of $\cT_n$ and the stated
frame normalization.  The identification is functorial in $A$.
\end{proof}

\section{Artinian reconstruction}
\label{sec:artinian-reconstruction}

In this section we prove a formal deformation version of the reconstruction
theorem from Part~\ref{part:field}.  Let $A$ be a local Artinian
$\Z[1/n]$-algebra, let
\[
        k:=A/\mathfrak m_A,
\]
and let
\[
        (\widetilde P_i)_{i\in\Z/n\Z}\subset \PP^2_A
\]
be a normalized $A$-valued deformation of a $k$-realization of $\cT_n$.

\subsection{The relative seed pencil}

\begin{lemma}[Relative seed pencil]
\label{lem:relative-seed-pencil}
Define relative lines
\[
\begin{aligned}
L_1&=\widetilde P_0\widetilde P_4\widetilde P_{-4},&
L_2&=\widetilde P_1\widetilde P_2\widetilde P_{-3},&
L_3&=\widetilde P_{-1}\widetilde P_{-2}\widetilde P_3,\\
M_1&=\widetilde P_0\widetilde P_2\widetilde P_{-2},&
M_2&=\widetilde P_1\widetilde P_3\widetilde P_{-4},&
M_3&=\widetilde P_{-1}\widetilde P_{-3}\widetilde P_4.
\end{aligned}
\]
Let
\[
        D_1=L_1\cup L_2\cup L_3,
        \qquad
        D_2=M_1\cup M_2\cup M_3.
\]
Then the scheme-theoretic intersection $Z=D_1\cap D_2$ is the disjoint union of the nine sections
\[
        \widetilde P_{-4},\widetilde P_{-3},\ldots,\widetilde P_4,
\]
and the $A$-module of cubic forms vanishing on $Z$ is free of rank two, generated by $D_1$ and $D_2$.  Moreover, there is a unique relative cubic
$\mathcal C\subset \PP^2_A$ passing through
\[
        \widetilde P_{-4},\widetilde P_{-3},\ldots,\widetilde P_5.
\]
\end{lemma}

\begin{proof}
On the special fiber, the six row and column lines are pairwise distinct and form the $3\times3$ grid
\[
\begin{array}{c|ccc}
        & M_1&M_2&M_3\\ \hline
L_1& P_0&P_{-4}&P_4\\
L_2& P_2&P_1&P_{-3}\\
L_3& P_{-2}&P_3&P_{-1}.
\end{array}
\]
The nonvanishing determinants expressing these distinctness and transversality conditions are nonzero modulo the maximal ideal, hence they are units in $A$.  Therefore the relative intersections $L_i\cap M_j$ are exactly the nine displayed disjoint sections.  At each point of their union, the reductions of local equations for $D_1$ and $D_2$ are, up to units, equations of a transverse row-line and column-line.  Since $\mathcal O_{\PP^2_A}$ is flat over $A$, two successive applications of the slicing criterion \cite[Tag~00ME]{Stacks} show that these equations form a regular sequence and that their quotient is $A$-flat.  Hence
\[
        Z=D_1\cap D_2=\coprod_{i=-4}^4 \widetilde P_i
\]
is a relative complete intersection of two cubics.

The ideal sheaf $\mathcal I_Z$ has the Koszul resolution
\[
0\longrightarrow \mathcal O_{\PP^2_A}(-6)
 \longrightarrow \mathcal O_{\PP^2_A}(-3)^{\oplus2}
 \longrightarrow \mathcal I_Z
 \longrightarrow 0.
\]
Twisting by $\mathcal O(3)$ gives
\[
0\longrightarrow \mathcal O_{\PP^2_A}(-3)
 \longrightarrow \mathcal O_{\PP^2_A}^{\oplus2}
 \longrightarrow \mathcal I_Z(3)
 \longrightarrow 0.
\]
Since
\[
H^0(\PP^2_A,\mathcal O(-3))=H^1(\PP^2_A,\mathcal O(-3))=0,
\]
taking global sections gives
\[
        H^0(\PP^2_A,\mathcal I_Z(3))\cong A^{\oplus2},
\]
with generators represented by $D_1$ and $D_2$.

Evaluation at the section $\widetilde P_5$ gives an $A$-linear map
\[
        A^{\oplus2}\longrightarrow \widetilde P_5^*\mathcal O_{\PP^2_A}(3)\cong A.
\]
Modulo the maximal ideal this map is nonzero, because in the special fiber $P_5$ does not lie in $D_1\cap D_2$.  Hence the map is surjective.  Its kernel is therefore a projective $A$-module of rank one, which is in fact free of rank one since $A$ is local.  A generator of the kernel gives a relative cubic $\mathcal C$ through the ten sections, unique up to multiplication by a unit.  
\end{proof}

\begin{remark}\label{rem:relative-cubic-flat}
The defining form of the relative cubic $\mathcal C$ has nonzero reduction on
$\PP^2_k$, and hence is a nonzerodivisor in every local ring of
$\PP^2_k$.  Since $\mathcal O_{\PP^2_A}$ is flat over $A$, the slicing
criterion \cite[Tag~00ME]{Stacks} shows that it remains a nonzerodivisor over
$A$ and that $\mathcal C$ is flat over $A$.  Thus $\mathcal C$ is a relative
effective Cartier divisor on $\PP^2_A$.
\end{remark}

\subsection{Relative Chasles propagation}

\begin{lemma}[Relative Chasles completion]
\label{lem:relative-chasles-completion}
Let $Z$ be a $3\times3$ grid of disjoint sections of $\PP^2_A$, obtained as the transverse intersection of two reducible relative cubics $D_1$ and $D_2$, each a union of three relative lines.  Let $Z'$ be the union of any eight of the nine sections.  Then every relative cubic containing $Z'$ contains all of $Z$.
\end{lemma}

\begin{proof}
Let
\[
        V:=H^0(\PP^2_A,\mathcal O(3)),
        \qquad
        E_Z:=H^0(Z,\mathcal O_Z(3)),
        \qquad
        E_{Z'}:=H^0(Z',\mathcal O_{Z'}(3)).
\]
Since $Z$ and $Z'$ are disjoint unions of sections, $E_Z\cong A^9$ and $E_{Z'}\cong A^8$ are finite free $A$-modules, and formation of these modules commutes with reduction to the residue field.  Let
\[
        K_Z:=\ker(V\to E_Z),
        \qquad
        K_{Z'}:=\ker(V\to E_{Z'}).
\]
By the complete-intersection calculation used in \cref{lem:relative-seed-pencil}, $K_Z$ is a free $A$-module of rank $2$, generated by $D_1$ and $D_2$, and its formation commutes with base change.

On the special fiber, the eight points of $Z'_k$ impose eight independent conditions on cubics: equivalently, the space of cubics through them has dimension $2$, by the classical Chasles theorem for the reduced $3\times3$ grid.  Hence the evaluation map
\[
        V\otimes_A k\longrightarrow E_{Z'}\otimes_A k
\]
is surjective.  It follows from Nakayama's lemma that $V\to E_{Z'}$ is surjective.  Thus $K_{Z'}$ is a projective, hence free, $A$-module of rank $2$, and its formation commutes with base change.

We have an inclusion $K_Z\subseteq K_{Z'}$.  After tensoring with $k$, this inclusion becomes an equality, again by the classical Chasles theorem on the special fiber.  Therefore the cokernel $K_{Z'}/K_Z$ is a finite $A$-module 
whose special fiber is zero. Nakayama's lemma gives $K_{Z'}=K_Z$, and thus every cubic through $Z'$ is a cubic through $Z$, as claimed.
\end{proof}

\begin{lemma}[Relative propagation]
\label{lem:relative-propagation}
Let $\cC$ be the relative cubic of
Lemma~\ref{lem:relative-seed-pencil}.  Then every marked section
$\widetilde P_i$ lies on $\cC$.
\end{lemma}

\begin{proof}
The proofs of Lemmas~\ref{lem:odd-first-bootstrap},
\ref{lem:odd-propagation-step}, \ref{lem:even-first-bootstrap},
\ref{lem:even-propagation-step}, and~\ref{lem:half-period-grid} apply verbatim,
with Lemma~\ref{lem:relative-chasles-completion} in place of
Lemma~\ref{lem:grid-completion}.  Indeed, every determinant used to certify
that a grid consists of nine distinct transverse sections is nonzero on the
special fiber and hence is a unit in $A$.  
(For $n=10$ the seed window already contains every residue class, and for $n=11$ the single bootstrap grid from
Subsection~\ref{subsec:n11} supplies the remaining section.)
\end{proof}

\subsection{Relative group-law recovery}

We use the following Picard-theoretic version of the group law on a generalized elliptic curve.

\begin{lemma}[Relative Picard form of the chord law]
\label{lem:relative-picard-chord}
Let $A$ be a local Artinian ring, and let $\mathcal C\subset\PP^2_A$ be a flat family whose special fiber is a smooth or irreducible nodal plane cubic.  Let $O$ be a section through the smooth locus.  Then $\mathcal C$ is a generalized elliptic curve over $A$, its smooth locus $\mathcal C^{\rm sm}$ is a commutative group scheme with identity $O$, and the Abel--Jacobi map identifies $\mathcal C^{\rm sm}$ with $\Pic^0_{\mathcal C/A}$.  If $Q,R,S$ are smooth sections with pairwise distinct special fibers, and if $H$ denotes the divisor class of a relative line section, then $Q,R,S$ are collinear if and only if
\[
        [Q-O]+[R-O]+[S-O]=[H-3O]
\]
in $\Pic^0_{\mathcal C/A}(A)$.
\end{lemma}

\begin{proof}
Since the special fiber is smooth or irreducible nodal, $\mathcal C/A$ is a stable genus-one curve with geometrically integral fibers.  Deligne--Rapoport \cite[II, Proposition~2.7]{DR} therefore gives a unique generalized-elliptic-curve structure with identity $O$ and identifies the Abel--Jacobi map
\[
\mathcal C^{\rm sm}\longrightarrow\Pic^0_{\mathcal C/A},
\qquad Q\longmapsto[Q-O]
\]
with an isomorphism of commutative group schemes.

It remains to verify the divisor computation.  Since the special fibers of $Q,R,S$ are pairwise distinct and lie in the smooth locus, $D:=Q+R+S$ is a relative effective Cartier divisor of degree $3$ on $\mathcal C$.  If $Q,R,S$ are collinear, then $D$ is cut out by a relative line in $\PP^2_A$, so $\mathcal O_{\mathcal C}(D)\cong\mathcal O_{\mathcal C}(H)$, and subtracting $3O$ gives the displayed identity.

Conversely, suppose the displayed identity holds.  Then $\mathcal O_{\mathcal C}(D)\cong\mathcal O_{\mathcal C}(H)$.  The exact sequence
\[
0\longrightarrow \mathcal O_{\PP^2_A}(-2)
\longrightarrow \mathcal O_{\PP^2_A}(1)
\longrightarrow \mathcal O_{\mathcal C}(H)
\longrightarrow 0
\]
and the vanishing $H^0(\PP^2_A,\mathcal O(-2))=H^1(\PP^2_A,\mathcal O(-2))=0$ show that restriction induces an isomorphism
\[
        H^0(\PP^2_A,\mathcal O(1))
        \xrightarrow{\sim}
        H^0(\mathcal C,\mathcal O_{\mathcal C}(H)).
\]
The divisor $D$ therefore comes from a unique relative line in $\PP^2_A$, and hence the three sections are collinear.
\end{proof}

\begin{lemma}[Relative group-law recovery]
\label{lem:relative-group-law-recovery}
Let $\mathcal C$ be the relative cubic constructed above.  Then all sections $\widetilde P_i$ lie in $\mathcal C^{\rm sm}(A)$, and in $\Pic^0_{\mathcal C/A}(A)$ one has
\[
        [\widetilde P_i-\widetilde P_0]
        = i\cdot [\widetilde P_1-\widetilde P_0]
\]
for all $i\in\Z/n\Z$.  Moreover,
\[
        \mathcal O_{\mathcal C}(1)\cong
        \mathcal O_{\mathcal C}(3\widetilde P_0),
\]
and $[\widetilde P_1-\widetilde P_0]$ has exact order $n$.
\end{lemma}

\begin{proof}
By Remark~\ref{rem:relative-cubic-flat}, $\mathcal C$ is a flat relative effective Cartier divisor on $\PP^2_A$, and its special fiber is the irreducible smooth-or-nodal cubic reconstructed in Part~\ref{part:field}.  All marked points of the special fiber lie in its smooth locus.  Since smoothness is open and each $\widetilde P_i$ is a section specializing to a smooth point, all $\widetilde P_i$ lie in $\mathcal C^{\rm sm}(A)$.

Let $H$ denote the class of a relative line section, and set
\[
        x_i:=[\widetilde P_i-\widetilde P_0]
        \in \Pic^0_{\mathcal C/A}(A),
        \qquad
        \lambda:=[H-3\widetilde P_0].
\]
Whenever $a,b,c$ are distinct and satisfy $a+b+c=0$ in $\Z/n\Z$, the representation condition gives
\[
\det(v_a,v_b,v_c)=0
\]
for unimodular homogeneous lifts $v_a,v_b,v_c\in A^3$.  The cross product
$w=v_a\times v_b$ annihilates all three lifts.  Moreover, $w$ is unimodular:
its coordinates are the $2\times2$ minors of $(v_a\mid v_b)$, and one of
these minors is a unit because the special fibers of $\widetilde P_a$ and
$\widetilde P_b$ are distinct.  Thus $w$ defines a relative line containing
the three sections.  By \cref{lem:relative-picard-chord},
\[
        x_a+x_b+x_c=\lambda.
\]
Applying Lemma~\ref{lem:abstract-group-law} in the abelian group $\Pic^0_{\mathcal C/A}(A)$ gives
\[
        \lambda=0
        \qquad\text{and}\qquad
        x_i=i\cdot x_1
\]
for all $i\in\Z/n\Z$.

The equality $\lambda=0$ is exactly
\[
        \mathcal O_{\mathcal C}(1)\cong
        \mathcal O_{\mathcal C}(3\widetilde P_0).
\]
Since the indices are periodic modulo $n$, we have $x_n=x_0=0$, so $n x_1=0$.  If $m x_1=0$ for some integer $0<m<n$, then $d x_1=0$ for $d=\gcd(m,n)$; here $d$ is a proper divisor of $n$.  Reducing modulo the maximal ideal of $A$ would give $d x_{1,k}=0$ on the
special fiber, contradicting the exact order $n$ statement from
Part~\ref{part:field}.  Thus $x_1$ has exact order $n$.  Since $n$ is invertible in $A$, the group scheme $\mathcal C^{\rm sm}[n]$
is finite \etale{} over $A$.  The homomorphism
\[
(\Z/n\Z)_A\longrightarrow\mathcal C^{\rm sm}[n],
\qquad 1\longmapsto\widetilde P_1
\]
is a closed immersion because its special fiber is injective.  Since the
special fiber of $\mathcal C$ is irreducible, the resulting cyclic subgroup
meets every irreducible component.  Hence $\widetilde P_1$ is a
$\Gamma_1(n)$-structure in the sense of
\cite[Definition~2.4.1]{Conrad-KM}.
\end{proof}

\subsection{The Artinian reconstruction theorem}

\begin{theorem}[Artinian reconstruction]
\label{thm:artinian-reconstruction}
Let $\bar y=(C,O,P)\in X_1(n)^\circ(k)$ and let $\bar p$ be the
corresponding normalized representation of $\cT_n$.  For every local
Artinian $\Z[1/n]$-algebra $A$ with residue field $k$, the construction
\[
        (\mathcal C,\widetilde O,
        \widetilde P)
        \longmapsto
        (i\widetilde P)_{i\in\Z/n\Z}
\]
followed by the projective normalization of \cref{lem:projective-normalization-artinian} gives a bijection between:
\begin{enumerate}[label=\textup{(\roman*)}]
\item $A$-valued deformations of $\bar y$ in $X_1(n)^\circ$;
\item normalized $A$-valued representation deformations of $\bar p$.
\end{enumerate}
\end{theorem}

\begin{proof}
Starting from an $A$-valued deformation of the marked cubic, choose a trivialization of the free rank-three $A$-module $H^0(\mathcal C,\mathcal O_{\mathcal C}(3\widetilde O))$
to embed $\mathcal C$ in $\PP^2_A$.  The labelled torsion sections $i\widetilde P$ then form an $A$-valued representation of $\cT_n$.  A different trivialization changes the configuration by an element of $\PGL_3(A)$, and since the ordered quadruple $(\widetilde P_0,\widetilde P_1,\widetilde P_2,\widetilde P_3)$ remains a projective frame, the unique normalization of \cref{lem:projective-normalization-artinian} is independent of this choice.

Conversely, start with a normalized $A$-valued representation.  By \cref{lem:relative-seed-pencil} there is a unique relative cubic $\mathcal C$ through the seed window.  By \cref{lem:relative-propagation}, all marked sections lie on $\mathcal C$.  By \cref{lem:relative-group-law-recovery}, the section $\widetilde P_1$ has exact order $n$, the plane embedding is induced by $|3\widetilde P_0|$, and $\widetilde P_i=i\cdot\widetilde P_1$ in the generalized elliptic curve group law.  Thus $(\mathcal C,\widetilde P_0,\widetilde P_1)$ is an $A$-valued deformation of $\bar y$ in $X_1(n)^\circ$.

The two constructions are inverse.  In one direction this is immediate from the equality $\widetilde P_i=i\cdot\widetilde P_1$.  In the other direction, the relative cubic reconstructed from a normalized representation is unique through the seed window, so it must coincide with the original cubic.
\end{proof}

\begin{corollary}[Artinian-point bijectivity]
\label{cor:artinian-point-bijectivity}
For every local Artinian $\Z[1/n]$-algebra $A$, the modular morphism induces
a bijection
\[
        \beta(A):X_1(n)^\circ(A)\xrightarrow{\sim}\cR_n(A).
\]
\end{corollary}

\begin{proof}
Let $k=A/\mathfrak m_A$.  Reduction modulo $\mathfrak m_A$ gives maps
\[
        X_1(n)^\circ(A)\longrightarrow X_1(n)^\circ(k),
        \qquad
        \cR_n(A)\longrightarrow\cR_n(k)
\]
which are compatible with $\beta$.  The map
$\beta(k):X_1(n)^\circ(k)\longrightarrow\cR_n(k)$
is bijective by Corollary~\ref{cor:geometric-bijective}.

Fix $\bar y\in X_1(n)^\circ(k)$, let
$\bar x=\beta(\bar y)\in\cR_n(k)$, and let $\bar p$ be the corresponding
normalized realization.  By
Proposition~\ref{prop:canonical-normalized-deformations}, the fiber of
$\cR_n(A)\longrightarrow\cR_n(k)$
over $\bar x$ is naturally identified with the set of normalized
$A$-valued deformations of $\bar p$.  By
Theorem~\ref{thm:artinian-reconstruction}, this set is naturally identified
with the fiber of
$X_1(n)^\circ(A)\longrightarrow X_1(n)^\circ(k)$ 
over $\bar y$, and the identification is induced by $\beta$.

Thus $\beta(A)$ is bijective on the fibers over every point of the 
field-valued correspondence.  Since $\beta(k)$ is bijective, $\beta(A)$ is
bijective.
\end{proof}

\subsection{The global isomorphism}
\label{sec:global-isomorphism}

\begin{proof}[Proof of Theorem~\ref{thm:main}]
Both $X_1(n)^\circ$ and $\cR_n$ are of finite type over $\Z[1/n]$.
Apply Proposition~\ref{prop:artinian-point-criterion} to
$\beta:X_1(n)^\circ\longrightarrow\cR_n$
using Corollary~\ref{cor:artinian-point-bijectivity}.
\end{proof}

\section{Complement on modular curves and representations of elliptic matroids}
\label{sec:complement}

We begin with the following well-known description of the cusps of $X_1(n)$ which are ``missing''  from $X_1(n)^\circ$.

\begin{proposition}
\label{prop:cusps}
The complement $X_1(n)\setminus X_1(n)^\circ$ consists of the cusps represented
by reducible N\'eron polygons.  Over $\overline\Q$, the geometric points of
$X_1(n)^\circ\setminus Y_1(n)$ are represented by a N\'eron $1$-gon together with a primitive $n$-th root of
unity in its smooth locus, modulo inversion.  They are therefore indexed by
$(\Z/n\Z)^\times/\{\pm1\}$.
The Galois action is transitive, and the corresponding closed point on the
generic fiber has residue field $\Q(\zeta_n+\zeta_n^{-1})$.
In particular, none of the cusps in $X_1(n)^\circ\setminus Y_1(n)$ is
$\Q$-rational when $n\ge10$.
\end{proposition}

\begin{proof}
The cusps of $X_1(n)$ classify N\'eron polygons equipped with a point of exact
order $n$ in the smooth locus.  The fiber is irreducible exactly for a
N\'eron $1$-gon.  Its smooth locus is $\G_m$, so a point of exact order $n$
is a primitive $n$-th root of unity.  The automorphisms of the N\'eron
$1$-gon as a generalized elliptic curve restrict on $\G_m$ to the identity
and inversion.  Hence two primitive roots determine isomorphic marked
$1$-gons exactly when they differ by inversion.

The absolute Galois group acts on primitive roots through
$(\Z/n\Z)^\times$, transitively modulo $\{\pm1\}$.  The stabilizer of the
class of $\zeta_n$ is $\{\pm1\}$, whose fixed field is the maximal real
cyclotomic subfield $\Q(\zeta_n+\zeta_n^{-1})$.  Its degree is
$\varphi(n)/2>1$ for $n\ge10$.
\end{proof}

\begin{corollary}[Mazur's theorem and elliptic matroids]
\label{cor:mazur-matroid}
For every prime $p\ge11$, the elliptic matroid $\cT_p$ is not representable
over $\Q$.  Moreover, this assertion for all primes $p\ge11$ is equivalent to
Mazur's theorem that $Y_1(p)(\Q) = \emptyset$.
\end{corollary}

\begin{proof}
By Theorem~\ref{thm:field-valued-bijection}, a $\Q$-representation of
$\cT_p$ is equivalent to a point of $X_1(p)^\circ(\Q)$.  By
Proposition~\ref{prop:cusps}, the boundary
$X_1(p)^\circ\setminus Y_1(p)$ has no $\Q$-rational point.  Thus
\[
        X_1(p)^\circ(\Q)=Y_1(p)(\Q),
\]
and the latter set parametrizes elliptic curves over $\Q$ with a rational
point of order $p$.  Mazur's theorem says that no such point exists for
$p\ge11$ \cite{Mazur}.  Conversely, nonrepresentability of $\cT_p$ over
$\Q$, together with the field-valued correspondence and the boundary
calculation, implies the same prime-order assertion.
\end{proof}

\appendix

\section{The canonical realization band scheme of a matroid}
\label{sec:bands}

We construct the realization space $\bR_M$ of a matroid $M$ as a band scheme in the sense of \cite{BJ-bands}.
We then show that the scheme over $\Z$ associated to $\bR_M$ coincides with the realization scheme $\cR_M$ 
defined in the body of the paper.

We recall only the definitions and facts needed below; see \cite{BJ-bands,BL-Lafforgue} for
the general theory of bands and band schemes.

\subsection{Bands, pastures, and associated rings}

A \emph{pointed monoid} is a commutative multiplicative monoid $B$ with
identity $1$ and an absorbing element $0$.  Its ambient semiring is
\[
        B^+:=\N[B]/\langle 0\rangle,
\]
the semiring of finite formal sums of nonzero elements of $B$.

\begin{definition}\label{def:band}
A \emph{band} is a pointed monoid $B$ together with an ideal
$N_B\subset B^+$, called its \emph{null set}, such that for every $a\in B$
there is a unique element $-a\in B$ with $a+(-a)\in N_B$.  A morphism of
bands is a multiplicative map preserving $0$ and $1$ and carrying null
relations to null relations.
\end{definition}

When there is no danger of confusion, we write
\[
        a_1+\cdots+a_m = 0
\]
instead of $a_1+\cdots+a_m\in N_B$.  

Every commutative ring $R$ is likewise a band with its usual multiplicative
monoid and with
\[
        N_R=\left\{\sum a_i:\sum a_i=0\text{ in }R\right\}.
\]
This gives a fully faithful embedding of rings into bands. 

The field-like objects called pastures (of which fields, partial fields, and hyperfields are examples)
used in \cite{BL-foundations,BL-Lafforgue} and related works can also be thought of as bands.
More formally, there is a fully faithful inclusion
\[
        \Past\lhook\joinrel\longrightarrow\Band
\]
from pastures to bands.
We will identify both pastures and rings with their images in $\Band$.

\begin{definition}[Associated ring]\label{def:associated-ring-band}
Let $B$ be a band.  Its associated ordinary ring is
\[
 \rho_{\Z}(B):=
 \Z[B]\Big/\Big\langle \sum [a_i]\ \Big|\ \sum a_i\in N_B\Big\rangle,
\]
where $\Z[B]$ is the monoid ring of the pointed monoid $B$, with $[0]=0$.
\end{definition}

\begin{lemma}\label{lem:associated-ring-universal}
For every commutative ring $R$, regarded as a band, composition with the
canonical map $B\to\rho_{\Z}(B)$ gives a natural bijection
\[
 \Hom_{\mathrm{Ring}}(\rho_{\Z}(B),R)
 \cong
 \Hom_{\Band}(B,R).
\]
\end{lemma}

\begin{proof}
A band morphism $B\to R$ extends uniquely to a ring homomorphism
$\Z[B]\to R$.  It factors through $\rho_{\Z}(B)$ precisely because every
null relation in $B$ is sent to an ordinary additive relation in $R$.  Moreover, this
construction is easily verified to be reversible.
\end{proof}

We write $\Spec B$ for the affine band scheme representing the functor
$P\mapsto\Hom_{\Band}(B,P)$.  Its base extension to ordinary schemes is
$\Spec\rho_{\Z}(B)$.  These affine constructions are all that we need from
the general theory of band schemes.

\subsection{The universal band}

Let $E$ be a finite totally ordered set, and let $M$ be a matroid of rank $r \geq 2$ on $E$.
(The evident analogous definitions in ranks $0$ and $1$, where there are no Pl\"ucker relations, will be omitted here for simplicity; 
these are degenerate and uninteresting cases from the point of view of realization spaces.)
Write $\mathcal B(M)$ for its set of bases.  
For an $(r-2)$-subset $I$ and distinct elements $i,j,k,l\notin I$, we use the usual
alternating convention for symbols $p_{Iij}$, and set $p_J=0$ when $J$ is not
a basis of $M$.

\begin{definition}
\label{def:fixed-support-plucker-band}
The \emph{universal band} $U_M$ is the band generated by units
$p_B$, one for each $B\in\mathcal B(M)$, subject to the three-term Pl\"ucker
relations
\[
 p_{Iij}p_{Ikl}-p_{Iik}p_{Ijl}+p_{Iil}p_{Ijk} = 0.
\]
\end{definition}

Changing the auxiliary total order on $E$ only changes alternating symbols
by the corresponding signs and yields a canonical isomorphism.  Thus $U_M$
and the constructions derived from it are independent of this order, up to
canonical isomorphism.

\begin{lemma}\label{lem:UM-represents-fixed-support}
For every pasture $P$, morphisms
\[
U_M\longrightarrow P
\]
are naturally in bijection with weak Grassmann--Pl\"ucker functions
over $P$ having support $\mathcal B(M)$.
\end{lemma}

\begin{proof}
A morphism assigns a unit of $P$ to every basis symbol, assigns zero
to every nonbasis symbol, and carries the defining three-term
Pl\"ucker relations to null relations in $P$. These are precisely the
axioms for a weak Grassmann--Pl\"ucker function with support
$\mathcal B(M)$.
\end{proof}

\subsection{The degree-zero band and its associated scheme}

Let
\[
        G_M:=\Z\oplus\Z^E
\]
and grade $U_M$ by
\[
        \mdeg(p_B)=\left(1,\sum_{e\in B}e_e\right).
\]
The first coordinate records global scaling and the remaining coordinates
record (the band-theoretic analogue of) column rescaling.  

\begin{definition}
\label{def:degree-zero-band}
The \emph{realization band} of $M$ is the multidegree-zero sub-band
\[
        R_M :=(U_M)_0.
\]
\end{definition}

We retain the notation
\[
C_M,\qquad J_M,\qquad S_M=C_M/J_M,\qquad
A_M=(S_M)_0
\]
from Section~\ref{sec:schemes}.

\begin{lemma}\label{lem:associated-ring-fixed-support}
There is a natural isomorphism
\[
        \rho_{\Z}(U_M)\cong S_M.
\]
\end{lemma}

\begin{proof}
It is enough to compare the functors represented by the two rings.  Let $R$ be
a commutative ring.  A band morphism $U_M\to R$ assigns a unit of $R$ to each
basis coordinate and zero to each nonbasis coordinate, subject to the three-term
Pl\"ucker relations.  Its values therefore lie in the partial field
$R^\bullet=R^\times\cup\{0\}$ associated with $R$.  Partial fields are
perfect, so weak and strong Grassmann--Pl\"ucker functions over
$R^\bullet$ coincide \cite{BB}.  Consequently, band morphisms $U_M\to R$
are naturally the same as strong Grassmann--Pl\"ucker functions over $R$
with support $M$, which in turn are naturally the same as ring homomorphisms
$S_M\to R$.  The universal property of $\rho_{\Z}(U_M)$ and Yoneda's lemma
give the asserted isomorphism.
\end{proof}

\begin{lemma}\label{lem:associated-ring-degree-zero}
There is a natural isomorphism
\[
        \rho_{\Z}(R_M)\cong A_M.
\]
\end{lemma}

\begin{proof}
Let $D_M\subset G_M$ be the subgroup generated by the degrees of the basis
symbols.  Since $D_M$ is free, the exact sequence
\[
1\longrightarrow (R_M)^\times
\longrightarrow U_M^\times
\longrightarrow D_M\longrightarrow0
\]
admits a multiplicative splitting.  With such a splitting fixed, every
nonzero homogeneous element of $U_M$ can be written uniquely as a product of
a multidegree-zero element and a chosen element of its degree.  Moreover, every
defining Pl\"ucker relation is homogeneous, so division by the chosen
element of its common multidegree turns it into a defining multidegree-zero 
relation.  Thus the presentation of $U_M$ is the Laurent extension of
$R_M$ by the group $D_M$.  Applying the associated-ring functor and
using Lemma~\ref{lem:associated-ring-fixed-support} gives a graded
isomorphism
\[
        S_M\cong \rho_{\Z}(R_M)[D_M].
\]
Taking multidegree zero elements on both sides now proves the claim.
\end{proof}

\begin{definition}
\label{def:canonical-realization-scheme}
The canonical realization band scheme and its associated
ordinary scheme are
\[
        \bR_M:=\Spec R_M,
        \qquad
        \cR_M:=\Spec A_M.
\]
Thus $\cR_M$ is the base extension of $\bR_M$ from $\Fpm$ to $\Z$.
\end{definition}

\subsection{Foundations and reduced Dressians}

Let $P$ be a pasture. Two weak $P$-representations $\Delta$ and $\Delta'$ of $M$ are
\emph{rescaling equivalent} if there are units $u,t_e\in P^\times$ such
that
\[
        \Delta'(B)=u\left(\prod_{e\in B}t_e\right)\Delta(B)
        \qquad(B\in\mathcal B(M)).
\]
Let $\mathcal X_M(P)$ be the set of rescaling classes of weak $P$-representations of $M$.  By the work of
Baker--Lorscheid, this functor on pastures is represented by the foundation
$F_M$ of $M$:
\[
        \mathcal X_M(P)\cong\Hom_{\Past}(F_M,P)
\]
functorially in $P$ \cite{BL-moduli,BL-foundations}.

\begin{proposition}\label{prop:degree-zero-rescaling}
For every pasture $P$, restriction to multidegree zero induces a natural bijection
\[
        \mathcal X_M(P)
        \xrightarrow{\sim}
        \Hom_{\Band}(R_M,P).
\]
\end{proposition}

\begin{proof}
Let $D_M\subset G_M$ be the subgroup generated by the degrees of the basis
symbols.  We first note that $D_M$ is a direct summand of $G_M$.  Let
$E_1,\dots,E_c$ be the connected components of $M$, and put
$r_j=\rk(M|E_j)$.  Every basis $B$ contains exactly $r_j$ elements of $E_j$,
so $D_M$ is contained in the kernel of
\[
 \psi:\Z\oplus\Z^E\longrightarrow\Z^c,
 \qquad
 (a,x)\longmapsto
 \left(\sum_{e\in E_j}x_e-r_j a\right)_{j=1}^c.
\]
Conversely, fix a basis $B_0$.  It is not hard to check that differences of incidence vectors of bases
generate all vectors $e-e'$ with $e$ and $e'$ in the same connected component
of $M$.
Indeed, if $e$ and $e'$ belong to the same connected component, then
there is a circuit $C$ containing both.  Extend the independent set
$C\setminus\{e'\}$ to a basis $B$.  Then $B-e+e'$
is also a basis, and the difference of their incidence vectors is
$e'-e$.  Hence all such differences are generated by differences of
basis incidence vectors.

Thus, if $(a,x)\in\ker(\psi)$, then
\[
        (a,x)-a \cdot \mdeg(p_{B_0})
\]
has first coordinate zero and coordinate sum zero on every component, and is
therefore an integral combination of differences of basis degrees.  Hence
$D_M=\ker(\psi)$.  The map $\psi$ is surjective (use a coordinate vector
from each $E_j$), and its target is free.  Consequently $D_M$ is saturated and
is a direct summand of $G_M$.

The multidegree map gives an exact sequence of abelian groups
\[
 0\longrightarrow (R_M)^\times
 \longrightarrow U_M^\times
 \longrightarrow D_M\longrightarrow0.
\]
Since $D_M$ is free, choose a splitting.  If
$f:R_M\to P$ is a band morphism, send the chosen homogeneous
splitting elements to $1\in P^\times$ and use $f$ on the multidegree-zero factor.
Every defining 3-term Pl\"ucker relation is homogeneous, and after division by
the splitting element of its common multidegree, it becomes a null relation in
$R_M$.  Thus the extension sends the generators of the null set
of $U_M$ to null relations in $P$, and hence is a band morphism
$U_M\to P$.  Equivalently, it defines a weak $P$-representation of $M$.

Two such extensions have the same restriction to multidegree zero precisely when
their ratio is a character $D_M\to P^\times$.  Because $D_M$ is a direct
summand of $G_M$, this character extends to $G_M$.  A character of
$G_M=\Z\oplus\Z^E$ is exactly the choice of a global scalar and independent
column scalars, so the two representations are rescaling equivalent.
Conversely, every rescaling character is trivial on multidegree-zero elements.
This proves the claimed bijection.
\end{proof}

Proposition~\ref{prop:degree-zero-rescaling} explains the compatibility of our notation with Part~\ref{part:field}:
when $k$ is a field,
\[
R_M(k)
\cong
\Hom_{\Band}(R_M,k).
\]
Thus the set denoted by $R_M(k)$ in Part~\ref{part:field} is the
set of $k$-points of the realization band.

\begin{corollary}\label{cor:foundation-comparison}
For every pasture $P$, there is a natural bijection
\[
 \Hom_{\Band}(R_M,P)
 \cong
 \Hom_{\Past}(F_M,P).
\]
Equivalently, $F_M$ represents the restriction of the functor of points of
$\bR_M$ to the full subcategory of pastures.
\end{corollary}

\begin{proof}
Both sides represent the functor $\mathcal X_M$ on pastures.
\end{proof}

Let $\T$ denote the tropical hyperfield: its underlying set is
$\mathbb R_{\ge0}$, its multiplication is the usual one, and its hyperaddition is
\[
a\boxplus b=
\begin{cases}
\max\{a,b\},&a\ne b,\\
[0,a],&a=b.
\end{cases}
\]
As an application of the preceding comparison, we identify the $\T$-points of the realization band scheme with the reduced Dressian.

\begin{corollary}\label{cor:reduced-dressian}
There is a canonical identification
\[
        \bR_M(\T)
        =\Hom_{\Band}(R_M,\T)
        \cong\underline{\operatorname{Dr}}(M),
\]
where the right-hand side is the reduced Dressian of $M$, i.e., the set of
valuated matroids with underlying matroid $M$ modulo rescaling.
\end{corollary}

\begin{proof}
Weak matroids over the tropical hyperfield are the same thing as valuated matroids in the sense of Dress--Wenzel, cf.
\cite{BB}.  The foundation construction identifies
$\Hom_{\Past}(F_M,\T)$ with their rescaling classes
\cite{BL-moduli,BL-Lafforgue}.  Now apply
Corollary~\ref{cor:foundation-comparison}.
\end{proof}

\begin{remark}[Intrinsic tropicalization]\label{rem:intrinsic-tropicalization}
We define the \emph{intrinsic tropicalization} of the realization space associated to a matroid $M$ by
\[
        \operatorname{Trop}(\cR_M)
        :=\bR_M(\T).
\]
Corollary~\ref{cor:reduced-dressian} identifies the right-hand side canonically with
$\underline{\operatorname{Dr}}(M)$.  
\end{remark}

In particular, Theorem~\ref{thm:main} and
Remark~\ref{rem:intrinsic-tropicalization} endow $X_1(n)^\circ$ with the
intrinsic tropicalization $\bR_{\cT_n}(\T)$; the geometry of $\underline{\operatorname{Dr}}(\cT_n)$ (especially when $n=p$ is prime) 
will be explored in \cite{BES}.

\end{document}